\documentclass[11pt,reqno]{amsart}
\usepackage{amsmath,amssymb,amsthm,mathtools}
\usepackage{enumitem}
\usepackage[a4paper,margin=1.05in]{geometry}
\usepackage[hidelinks]{hyperref}
\usepackage{microtype}
\usepackage{bm}
\usepackage{mathrsfs}
\usepackage{cite}
\usepackage{needspace}

\newtheorem{theorem}{Theorem}[section]
\newtheorem{lemma}[theorem]{Lemma}
\newtheorem{proposition}[theorem]{Proposition}

\newtheorem{assumption}[theorem]{Assumption}
\newtheorem*{comparisonprinciple}{Comparison principle}
\newcommand{\paper}[3]{[#1, \emph{#2} #3]}
\newcommand{\R}{\mathbb{R}}

\newcommand{\loc}{\mathrm{loc}}
\newcommand{\cS}{\mathcal{S}}

\title[Convergence in higher dimensions]{Convergence for the porous medium equation with combustion in higher dimensions}
\author{Maolin Zhou}
\address{Chern Institute of Mathematics, Nankai University, Tianjin 300071, P. R. China}
\email{zhouml123@nankai.edu.cn}
\date{}
\begin{document}
\begin{abstract}
\paper{Lou--Zhou}{JEMS}{2026} proved convergence and sharp threshold
results for the one-dimensional porous medium equation with combustion
reaction. In this paper, we extend these results to radially symmetric
solutions with compactly supported initial data in higher dimensions.
Under suitable assumptions on the reaction term, we first prove that
every solution converges to a stationary solution. Based on this result,
we obtain a spreading--transition--vanishing trichotomy with a sharp
threshold. When $N=2$, transition solutions converge to the ignition
temperature $\theta$ locally uniformly in $\mathbb R^2$ as $t\to\infty$.
When $N\ge3$, they converge to positive ground state solutions uniformly
in $\mathbb R^N$ as $t\to\infty$. The main difficulty is to exclude
$\theta$ from the $\omega$-limit set when $N\ge3$. We prove this by
normalizing the difference $u-\theta$, estimating the resulting limit
with the heat kernel, and applying the radial zero number argument.
In contrast, when $N=1$ or $N=2$, the $\omega$-limit set of a
transition solution is $\{\theta\}$.
\end{abstract}
\maketitle

\section{Introduction}

Consider the reaction-diffusion equation
\[
 u_t=\Delta u+f(u),\qquad x\in\mathbb R^N,\quad t>0.
\]
For nonnegative compactly supported initial data, we are interested in
the asymptotic behavior of the solution as $t\to\infty$, in particular,
whether spreading or vanishing happens. Aronson and Weinberger
\cite{AronsonWeinberger1975,AronsonWeinberger1978} studied this
problem for monostable, bistable, and combustion reactions and gave
sufficient conditions for spreading and vanishing. For bistable and
combustion reactions, both spreading and vanishing can happen,
depending on the initial data.

In one dimension, Zlato\v{s} \cite{Zlatos2006} considered initial data
that are characteristic functions of intervals. He proved that a
critical interval length separates spreading from vanishing. At the
critical length, the solution converges to a ground state solution
in the bistable case and to the ignition temperature in the combustion
case. Du and Matano \cite{DuMatano2010} proved a general convergence
theorem for nonnegative compactly supported initial data and obtained
sharp thresholds for ordered families of initial data. Du and Lou
\cite{DuLou2015} established related results for nonlinear diffusion
problems with Stefan free boundary conditions.

Muratov and Zhong \cite{MuratovZhong2017} studied sharp thresholds
in higher dimensions by energy methods. They considered monotone
families of radially symmetric, nonincreasing initial data in
$L^2(\mathbb R^N)\cap L^\infty(\mathbb R^N)$.
For combustion reactions, they showed that the transition limit
depends on the dimension. Under their assumptions on the reaction
term and stationary solutions, the threshold solution converges
to the ignition temperature locally uniformly in $\mathbb R^N$
as $t\to\infty$ when $N\le2$, and to a positive ground state
solution uniformly in $\mathbb R^N$ as $t\to\infty$ when $N\ge3$.

In this paper, we study the Cauchy problem for the porous medium
equation with combustion reaction,
\begin{equation}\label{eq:PDE}
 \begin{cases}
 u_t=\Delta u^m+f(u),&x\in\mathbb R^N,\quad t>0,\\
 u(x,0)=u_0(x),&x\in\mathbb R^N,
 \end{cases}
 \qquad m>1,\quad N\ge2.
\end{equation}
The initial function $u_0$ is nonnegative, radially symmetric, and
compactly supported. The nonlinear diffusion in \eqref{eq:PDE}
occurs in models of gas flow through porous media and population
transport; see \cite{Vazquez2007,LouZhou2024} and the references
therein. Since $m>1$, the equation degenerates at $u=0$.
Solutions have finite propagation speed and free boundaries. They are classical where
they are positive, but the equation must be understood in the very
weak sense across the free boundaries. The existence, comparison,
and regularity theory can be found in \cite{Sacks1983,Vazquez2007}.

For the porous medium equation with reaction on a bounded interval,
Aronson, Crandall, and Peletier \cite{AronsonCrandallPeletier1982}
studied stabilization and characterized the stationary solutions
that can occur as limits. For the Cauchy problem in $\mathbb R^N$,
G\'arriz \cite{Garriz2020} gave sufficient conditions for spreading
and vanishing for several classes of reactions, including combustion
reactions. He also used traveling waves to describe the asymptotic
behavior of spreading solutions.

Lou and Zhou \cite{LouZhou2024} proved a general convergence theorem
for the one-dimensional porous medium equation with reactions.
Every bounded nonnegative global solution whose initial positive
set consists of finitely many bounded intervals converges to a
nonnegative zero of the reaction term or a ground state solution.
For combustion reactions, they obtained a
spreading--transition--vanishing trichotomy with a sharp threshold.
The transition solution converges to the ignition temperature
locally uniformly in $\mathbb R$ as $t\to\infty$.
Their proof uses the zero number and intersection number properties.
Intersections at the free boundaries require separate arguments because
the equation degenerates there.

Our aim is to study the asymptotic behavior of radially symmetric
solutions in dimensions $N\ge2$. We first classify the stationary
solutions and prove a general convergence result. We then prove the
spreading--transition--vanishing trichotomy and the sharpness of the
threshold. The transition limit is $\theta$ when $N=2$, and is a
positive ground state solution when $N\ge3$. The main difficulty is
to exclude $\theta$ from the $\omega$-limit set in higher dimensions.
For this purpose, we normalize $u-\theta$ and use the heat kernel to
estimate the solution of the limiting parabolic equation.

We impose the following assumptions on $f$.
\begin{assumption}\label{ass:reaction}
There exists $\theta\in(0,1)$ such that:
\begin{enumerate}[label=(F\arabic*)]
\item $f$ is locally Lipschitz on $[0,\infty)$ and belongs to
$C^1([\theta,\infty))\cap C^2((\theta,\infty))$, with
\[
 \begin{gathered}
 f(u)=0\quad(0\le u\le\theta),\qquad
 f(u)>0\quad(\theta<u<1),\\
 f(1)=0,\qquad f(u)<0\quad(u>1),\qquad
 f'_+(\theta)\ge0,\qquad f'(1)<0.
 \end{gathered}
\]
\item There is $\delta_\theta\in(0,1-\theta)$ such that $f$ is
nondecreasing on $[\theta,\theta+\delta_\theta]$.
\item If $N\ge3$, there is no bounded nonconstant positive radial
function $U$ such that $U^m\in C^2(\mathbb R^N)$ and
\[
 \Delta U^m+f(U)=0\quad\text{in }\mathbb R^N,
 \qquad \lim_{r\to\infty}U(r)=\theta.
\]
\end{enumerate}
\end{assumption}

Condition (F3) excludes regular radial stationary solutions tending to
$\theta$, without any assumption on their decay rate or energy.
We use it to classify the nonconstant limits and, together with (F2),
to prove the existence of ground state solutions and exclude
convergence to $\theta$ when $N\ge3$. Condition (F2) is used to
estimate $R(q)\alpha(q)$ in Lemma~\ref{lem:shooting-crossing}
and to compare two
solutions that converge to $\theta$ locally uniformly in $\mathbb R^N$
as $t\to\infty$.

For a radial function, we use the same notation for its expression
in $x$ and in $r=|x|$; thus $u(x,t)=u(r,t)$. The initial data
belong to the set
\begin{equation}\label{Initial-data}
\begin{split}
\mathscr X:=\bigl\{\phi\in C([0,\infty)):\;&
\text{there is }R_0>0\text{ such that }
\phi>0\text{ on }[0,R_0),\\
&\phi=0\text{ on }[R_0,\infty),\quad
\phi^{m-1}\in C^1([0,R_0]),\quad
(\phi^{m-1})'(R_0)<0\bigr\}.
\end{split}
\end{equation}
No monotonicity of $\phi$ on $[0,R_0]$ is assumed. For $N\ge3$, let
$\cS$ be the set of positive radial stationary solutions $U$ satisfying
\begin{equation}\label{eq:ground-state-definition}
 \Delta U^m+f(U)=0\quad\text{in }\mathbb R^N,\qquad
 U_r(r)<0\quad(r>0),\qquad \lim_{r\to\infty}U(r)=0.
\end{equation}
These are called ground state solutions. We will show that $\cS$ is
nonempty and that every $U\in\cS$ has $\theta<U(0)<1$.

We say that a solution \emph{vanishes} if it converges to zero
uniformly in $\mathbb R^N$ as $t\to\infty$, and that it
\emph{spreads} if it converges to one locally uniformly in
$\mathbb R^N$ as $t\to\infty$. A solution that neither vanishes
nor spreads is called a \emph{transition solution}.

We denote by $\omega(u)$ the set of functions $w$ such that
$u(\cdot,t_n)\to w$ locally uniformly in $\mathbb R^N$ for some
sequence $t_n\to\infty$. When $w$ is constant, we also use its
value to denote this constant function.

\Needspace{15\baselineskip}
\begin{theorem}\label{thm:sharp-threshold}
Suppose that Assumption~\ref{ass:reaction} holds. Let $\phi\in\mathscr X$,
and let $u_\sigma$ be the solution of \eqref{eq:PDE} with initial
value $\sigma\phi$. Then there exists a unique
$\sigma_*\in(0,\infty]$ such that the following conclusions hold.
\begin{enumerate}[label=(\roman*)]
\item \textup{Vanishing.} If $0<\sigma<\sigma_*$, then
$u_\sigma(\cdot,t)$ converges to zero uniformly in $\mathbb R^N$
as $t\to\infty$.
\item \textup{Spreading.} If $\sigma>\sigma_*$, then
$u_\sigma(\cdot,t)$ converges to one locally uniformly in
$\mathbb R^N$ as $t\to\infty$.
\item \textup{Transition.} If $\sigma_*<\infty$ and $N=2$, then
$u_{\sigma_*}(\cdot,t)$ converges to $\theta$ locally uniformly in
$\mathbb R^2$ as $t\to\infty$. If $\sigma_*<\infty$ and $N\ge3$,
then there exists $U_*\in\cS$ such that $u_{\sigma_*}(\cdot,t)$
converges to $U_*$ uniformly in $\mathbb R^N$ as $t\to\infty$.
\end{enumerate}
If $\sigma_*=\infty$, the solution $u_\sigma$ vanishes for every
finite $\sigma>0$. Moreover, if $N\ge3$, then
\begin{equation}\label{eq:main-omega-exclusion}
 \theta\notin\omega(u_\sigma)\qquad\text{for every }\sigma>0.
\end{equation}
\end{theorem}

We now explain the main ideas of the proof. In dimension two, every bounded
nonnegative radial stationary solution is constant. When $N\ge3$,
positive ground state solutions exist, and $U^m$ is harmonic
outside the set $\{U>\theta\}$. This classification alone does not
exclude convergence to $\theta$, since $\theta$ is a stationary
solution in every dimension.

The zero number argument is used first to prove that $u(0,t)$ has a
limit as $t\to\infty$. We then consider limits of $u(r,t_n+s)$
as $t_n\to\infty$. Their value at $r=0$ is independent of $s$.
When this value is positive, we compare the limit with the stationary
solution having the same central value. Their difference has a multiple
zero at the origin for every $s$, so the zero number diminishing
property implies that they are identical. When the central value is
zero, comparison with a solution of the porous medium equation gives
that the limit is identically zero. Uniqueness for the radial
stationary equation then gives convergence to a single stationary solution.

To exclude $\theta$ from the $\omega$-limit set when $N\ge3$,
we suppose that $u(\cdot,t)$ converges to $\theta$ locally uniformly
in $\mathbb R^N$ as $t\to\infty$. A weighted integral identity
and explicit subsolutions give time intervals of increasing length
during which $u>\theta$ on large annuli outside the ball containing
the initial support. We choose times in these intervals and divide
$u-\theta$ by its maximum absolute value on the fixed ball.
Passing to a subsequence, we obtain a nonzero limit $\eta$ satisfying
\[
 \eta_t=m\theta^{m-1}\Delta\eta+f'_+(\theta)\eta_+,
 \qquad \eta_+:=\max\{\eta,0\},
\]
which is nonnegative and nonincreasing in $r$ outside the same ball.

We next compare the positive and negative parts of $\eta$.
In the exterior domain, comparison with a solution of the heat
equation with a homogeneous Neumann boundary condition gives a
positive lower bound for the integral of $\eta_+$ at all later
times. The negative part remains inside the fixed ball.
Comparison with a Dirichlet eigenfunction shows that its integral
tends to zero. The heat kernel representation then shows that,
on every fixed ball, $\eta$ is strictly positive for all sufficiently
large times. Returning to the original solution, we find a time
at which the function $u(r,t)-\theta$ has exactly one zero.
We compare $u$ with a radial stationary solution tending to a
constant below $\theta$. The zero number diminishing property then
gives a contradiction to convergence to $\theta$.

Finally, we compare solutions with different initial amplitudes after
small spatial translations. In dimension two, we also change the
space and time variables and multiply the solution by a constant
slightly larger than one. These comparisons show that two different
amplitudes cannot both give transition solutions. This proves the
sharpness of the threshold.

The rest of this paper is arranged as follows.
Section~\ref{sec:preliminaries} gives the preliminary results.
Section~\ref{sec:stationary} classifies radial stationary solutions,
and Section~\ref{sec:convergence} proves the general convergence
theorem. In Section~\ref{sec:notheta-normalized}, we exclude $\theta$
from the $\omega$-limit set when $N\ge3$.
Sections~\ref{sec:vanishing-spreading}
and~\ref{sec:sharp-threshold} complete the proof of the sharp
threshold theorem.

\section{Preliminaries}\label{sec:preliminaries}

We use the notion of very weak solution as in \cite{LouZhou2024}.
A bounded nonnegative continuous function $u$ on
$\mathbb R^N\times[0,\infty)$ is a very weak solution of
\eqref{eq:PDE} if, for every $T>0$ and every smooth test function
$\varphi$ with compact spatial support in
$\mathbb R^N\times[0,T]$,
\begin{equation}\label{eq:very-weak}
\begin{split}
 &\int_{\mathbb R^N}u(x,T)\varphi(x,T)\,dx
 -\int_{\mathbb R^N}u_0(x)\varphi(x,0)\,dx\\
 &\hspace{8mm}=\int_0^T\!\!\int_{\mathbb R^N}
 \bigl(u\varphi_t+u^m\Delta\varphi+f(u)\varphi\bigr)\,dx\,dt.
\end{split}
\end{equation}
If the equality in \eqref{eq:very-weak} is replaced by $\le$
(respectively, $\ge$) for every nonnegative test function, then $u$
is called a very weak subsolution (respectively, supersolution).
All solutions below are understood in this sense.

We recall the comparison principle for very weak solutions; see
\cite{Sacks1983,Vazquez2007} and the formulation used in
\cite{LouZhou2024}.

\begin{comparisonprinciple}
Let $T>0$, and let $D$ be either $\mathbb R^N$ or a bounded domain
with $C^2$ boundary. Suppose that $f$ satisfies \textup{(F1)}, and
let $\underline u$ and $\overline u$ be bounded, nonnegative,
continuous functions on $\overline D\times[0,T]$ that are very weak
sub- and supersolutions, respectively, of
\[
 v_t=\Delta v^m+f(v)\qquad\text{in }D\times(0,T).
\]
If
\[
 \underline u(x,0)\le\overline u(x,0)\qquad(x\in\overline D),
\]
and, when $D$ is bounded,
\[
 \underline u(x,t)\le\overline u(x,t)
 \qquad(x\in\partial D,\ 0\le t\le T),
\]
then
\[
 \underline u(x,t)\le\overline u(x,t)
 \qquad(x\in\overline D,\ 0\le t\le T).
\]
\end{comparisonprinciple}

We also use the standard local H\"older estimates, continuous
dependence on the initial data at positive times, finite propagation,
and persistence of positivity; see \cite{Sacks1983,Vazquez2007}.
We also use the Zel'dovich--Kompaneets--Barenblatt (ZKB) solutions of
the porous medium equation without reaction; see \cite{Vazquez2007}.
The following proposition gives the basic properties needed below.

\begin{proposition}\label{prop:basic}
For $u_0\in\mathscr X$, problem \eqref{eq:PDE} has a unique bounded
nonnegative very weak solution $u(x,t)$, which is radial and is classical when $u(x,t)>0$. Moreover, $u(x,t)$ satisfies
\begin{equation}\label{eq:maximum-bound}
 0\le u(x,t)\le M:=\max\{1,\|u_0\|_\infty\},\qquad
 \limsup_{t\to\infty}\|u(\cdot,t)\|_\infty\le1.
\end{equation}
The positive set of $u(\cdot,t)$ is a ball $B_{b(t)}$, where
$b(t)$ is nondecreasing and
\begin{equation}\label{eq:support-divergence}
 b(t)\longrightarrow\infty\quad\text{as }\ t\to\infty.
\end{equation}
If $R_0$ is the initial support radius, then
\begin{equation}\label{eq:outermono}
 u_r(r,t)<0\qquad(R_0<r<b(t),\ t>0).
\end{equation}
\end{proposition}

\begin{proof}
Since $f$ is Lipschitz on bounded intervals, the existence, uniqueness and local H\"older regularity follow from \cite{Sacks1983,Vazquez2007}. If $u>0$ on compact set  $\Omega\in\mathbb R^N\times[0,\infty)$, Interior parabolic estimates shows
that $u(x,t)$ is classical on $\Omega$. By uniqueness, we have $u(x,t)$ is radial symmetry. Let $y(t)$ be the solution of $y'=f(y)$ satisfying $y(0)=M$,  a direct calculation yields that $0$ and $y(t)$ are respectively the subsolution and supersolution of \eqref{eq:PDE}, by comparison we have $0\leq u(x,t)\leq y(t)$. Since $f(u)<0$ for $u>1$, we obtain  \eqref{eq:maximum-bound}.

If $u(x_0,t_0)>0$, continuity allows us to choose a ZKB solution $z(x,t+\tau)$, centered at $x_0$ with $\tau>0$,  $z(\cdot,\tau)<u(\cdot,t_0)$ and $z(\cdot,\tau)<\theta$.
Since the maximum of $z(x,t+\tau)$ decreases with time, we have  $f(z(\cdot,t+\tau))=0$ for all $t>0$. By comparison, we have the persistence of positivity: if $u(x_0,t_0)>0$, then $u(x_0,t_0+t)\geq  z(x_0,\tau+t)>0$ for all $t>0$. Together with radial symmetry of $u(x,t)$, we have the positive set of $u(\cdot,t)$ is a ball $B_{b(t)}$ and $b(t)$ is nondecreasing. Since $u(0,0)>0$, we can choose a ZKB solution $z_1(x,t+\tau)$ centered at the origin with $\tau>0$, then $b(t)\to\infty \text{ as }t\to\infty$ due to support radius of $z_1(x,t+\tau)$ tends to infinity.

Fix a unit vector $e$ and $a>R_0$, and write
$x^a=x-2(x\cdot e-a)e$, then $u(x^a,t)$ is also the solution of \eqref{eq:PDE}. In the half-space $x\cdot e>a$,
$u_0(x)=0\le u_0(x^a)$, $u(x^a,t)$ and $u(x,t)$ agree on $x\cdot e=a$.
Comparison gives
\[
 u(x,t)\le u(x^a,t)\qquad(x\cdot e>a,\ t>0).
\]
The arbitrariness of $a$ gives $u_r(r,t)\leq0\ (R_0<r<b(t),\ t>0)$. To prove strict monotonicity, fix $t_*>0$ and
$R_0<a<b(t_*)$, and set
$D(x,t)=u^m(x^a,t)-u^m(x,t)$.   At $x_0=b(t_*)e$, the reflected
point lies strictly inside $B_{b(t_*)}$, so $D(x_0,t_*)>0$. By continuity, choose $a<r_1<\rho<b(t_*)$ such that
$D(r_1e,t_*)>0$. Since $u(\cdot,t_*)>0$ on $\overline B_\rho$, by continuity, there exists $\varepsilon\in(0,t_*)$ such that
$u$ has a positive lower bound on
$\overline B_\rho\times[t_*-\varepsilon,t_*]$.
In the domain $B_\rho\cap\{x\cdot e>a\}$, the reflected points
also lie in $B_\rho$. Hence, $D\ge0$ and $D\not\equiv0$ due to $D(r_1e,t_*)>0$, $D$ satisfies a uniformly parabolic
linear equation with bounded coefficients on $B_\rho\cap\{x\cdot e>a\}\times[t_*-\varepsilon,t_*]$.
The strong maximum principle gives $D>0$ in $B_\rho\cap\{x\cdot e>a\}$ for
$t_*-\varepsilon<t\le t_*$. Since $D=0$ on $x\cdot e=a$,
Hopf's boundary lemma at $(ae,t_*)$ gives
\[
 0<\partial_eD(ae,t_*)=-2(u^m)_r(a,t_*).
\]
The arbitrariness of $e$ and $a$ gives \eqref{eq:outermono}.
\end{proof}

\section{Radial stationary solutions}\label{sec:stationary}

Let $U=U(r)$ be a bounded nonnegative radial stationary solution of \eqref{eq:PDE}. Set
\[
 W=U^m,\qquad h(W)=f(W^{1/m}).
\]
By (F1), $h$ is locally Lipschitz on $(0,\infty)$ and vanishes on
$[0,\theta^m]$. Hence $h$ is locally Lipschitz on $[0,\infty)$.
Since $W$ satisfies $\Delta W=-h(W)$ in the weak sense, local
elliptic estimates give $W\in W^{2,p}_{\rm loc}(\R^N)$ for every
finite $p$. Sobolev embedding and elliptic regularity then imply
$W\in C^{2,\alpha}_{\rm loc}(\R^N)$ for every $0<\alpha<1$.
In radial coordinates, $W(r)$ satisfies
\begin{equation}\label{eq:W}
 W''+\frac{N-1}{r}W'+h(W)=0,
 \qquad W'(0)=0.
\end{equation}
Since $W'(0)=0$, multiplying \eqref{eq:W} by $r^{N-1}$ and integrating twice gives
\begin{equation}\label{eq:radial-integral}
 W(r)=W(0)-\int_0^r t^{1-N}
                   \int_0^t s^{N-1}h(W(s))\,ds\,dt.
\end{equation}
For two solutions $W_1$ and $W_2$ with $W_1(0)=W_2(0)$, let $L$ be a
Lipschitz constant for $h$ on their ranges. Set $M=\textup{max}_{r\in[0,R]}|W_1(r)-W_2(r)|$, by
\eqref{eq:radial-integral}, we have $M\leq MLR^{2}/(2N)$, thus $W_1(r)=W_2(r)$ for small $r$. The uniqueness theorem for
ordinary differential equations shows that $W_1(r)=W_2(r)$ for all $r\in I$ where $I$ is the common interval on which $W_1$ and $W_2$ exist, and by the continuous dependence for positive $r$, we have the continuous dependence
on $W(0)$.

\begin{lemma}\label{lem:stationarygeometry}
Every nonzero nonconstant bounded nonnegative radial stationary solution $U$ is
positive on $\mathbb R^N$ and satisfies $\theta<U(0)<1$.
In particular, it cannot have compact support.
\end{lemma}

\begin{proof}
By \eqref{eq:W}, when $W>1$,
\begin{equation}\label{x11}
 (r^{N-1}W')'=-r^{N-1}h(W)>0,
\end{equation}
thus if $U(0)>1$, $W'(r)>0$ for all $r>0$. Since $W$ is bounded, $W(r)\to L>1$ as $r\to \infty$. Due to  $h(L)<0$, \eqref{x11} shows that $W'(r)\ge cr$
for some $c>0$ and all sufficiently large $r$, we get a contradiction.
If $U(0)=1$ or $0\le U(0)\le\theta$, uniqueness for \eqref{eq:W} shows that $U$ is constant. Thus $\theta<U(0)<1$.

Suppose that $W(r_*)=0$ for some $r_*>0$. Nonnegativity and
$C^1$ regularity imply $W'(r_*)=0$. Since $h(0)=0$ and $h$ is
locally Lipschitz, uniqueness for \eqref{eq:W} at $r_*$ gives
$W\equiv0$ on $(0,\infty)$, by continuity, $W(0)=0$.
This contradicts $U(0)>\theta$. Hence $U(r)>0$ for every $r\ge0$.
\end{proof}

\subsection{The case \texorpdfstring{$N=2$}{N = 2}}

\begin{lemma}\label{lem:noGS2}
If $N=2$, every bounded nonnegative radial stationary solution is
constant and belongs to $[0,\theta]\cup\{1\}$. In particular,
there is no positive radial stationary solution tending to zero as
$r\to\infty$.
\end{lemma}

\begin{proof}
Suppose $U$ is nonconstant. Lemma~\ref{lem:stationarygeometry}
gives $U(0)=q\in(\theta,1)$. When $\theta^m<W<1$,
\[
 (rW')'=-rh(W)<0,
\]
thus $W'(r)<0$ for all $r>0$. We claim that $W$ reaches $\theta^m$ at a finite
radius. Otherwise, for any fixed $r_0>0$, when $r\ge r_0$,
\[
 rW'(r)\le-\int_0^{r_0}sh(W(s))\,ds=-c_0<0,
\]
then $W(r)\le W(r_0)-c_0\ln(r/r_0)$, a
contradiction. Let $R$ be the first radius where $W(R)=\theta^m$, we have $W'(R)<0$. Since $h=0$ on $[0,\theta^m]$, for $r\ge R$,
the function $W$ is harmonic and satisfies
\[
 W(r)=\theta^m+RW'(R)\ln\frac rR.
\]
Thus $W(r)$ reaches zero at a finite radius, contrary to Lemma~\ref{lem:stationarygeometry}.
\end{proof}

\subsection{The case \texorpdfstring{$N\ge3$}{N >= 3}}

We first show that every solution of \eqref{eq:W} with
$\theta^m<W(0)<1$ reaches $\theta^m$, and estimate its derivative
at that radius.

\begin{lemma}\label{lem:shooting-crossing}
Assume that $N\ge3$ and \textup{(F1)}, \textup{(F3)} hold.
For every $q\in(\theta,1)$,  the regular
radial solution  $W_q$ of
\eqref{eq:W} with $W_q(0)=q^m$ reaches $\theta^m$ at a finite radius $R(q)$ and $\alpha(q):=-W_q'(R(q))>0$.
If \textup{(F2)} also holds, then
\begin{equation}\label{eq:shooting-slope-bound}
 0<R(q)\alpha(q)\le2(q^m-\theta^m)
 \qquad(\theta<q\le\theta+\delta_\theta).
\end{equation}
In particular, under \textup{(F2)},
$R(q)\alpha(q)\to0$ as $q\downarrow\theta$.
\end{lemma}

\begin{proof}
When $W_q>\theta^m$, we have
\begin{equation}\label{x13}
 W_q'(r)=-r^{1-N}\int_0^r s^{N-1}h(W_q(s))\,ds<0.
\end{equation}
Suppose that $W_q$ never reaches $\theta^m$. Since $W_q'(r)<0$, $W(r)\to \ell\in[\theta^m,q^m)$ as $r\to \infty$.
If $\ell>\theta^m$, let $c>0$ be the minimum of $h$ on $[\ell,q^m]$,
we obtain $W_q'(r)\le-cr/N$ by \eqref{x13}, which
contradicts $W(r)\to \ell$ as $r\to \infty$. Hence $\ell=\theta^m$.
The function $U=W_q^{1/m}$ is then a bounded nonconstant positive
regular radial stationary solution tending to $\theta$, contrary
to \textup{(F3)}. Thus $R(q)<\infty$ and \eqref{x13} shows that $\alpha(q)>0$.

Assume now \textup{(F2)} and $q\le\theta+\delta_\theta$. Write
$R=R(q)$ and $g(r)=h(W_q(r))$ for $0\le r\le R$.
Since $W_q^{1/m}$ decreases from $q$ to $\theta$, the function $g$
is nonincreasing.  For $0<r\le R$, put
\[
 P(r)=-\frac{W_q'(r)}r
     =r^{-N}\int_0^r s^{N-1}g(s)\,ds.
\]
Monotonicity of $g$ implies $NP(r)\ge g(r)$ and consequently
\[
 P'(r)=\frac{g(r)-NP(r)}r\le0.
\]
Therefore
\[
 q^m-\theta^m
 =\int_0^R rP(r)\,dr
 \ge P(R)\frac{R^2}{2}
 =\frac{R\alpha(q)}2.
\]
Thus we have \eqref{eq:shooting-slope-bound} and $R(q)\alpha(q)\to0$ as $q\downarrow\theta$.
\end{proof}

Fix $q\in(\theta,1)$ and write $W(\cdot;q)=W_q$.
By Lemma~\ref{lem:shooting-crossing}, $W(\cdot;q)$ first reaches $\theta^m$ at a finite radius $R(q)$ and
\[
 \alpha(q)=-W_r(R(q);q)>0.
\]
For $r\ge R(q)$, the harmonic function with the same value and derivative at $R(q)$ is
\begin{equation}\label{eq:exterior}
 W(r)=\theta^m-\frac{\alpha(q)R(q)}{N-2}
 \left[1-\left(\frac{R(q)}r\right)^{N-2}\right],
\end{equation}
which decays to zero  when
\begin{equation}\label{eq:matching}
 R(q)\alpha(q)=(N-2)\theta^m.
\end{equation}

\begin{proposition}\label{prop:stationary}
Assume \textup{(F1)}, \textup{(F2)} and \textup{(F3)}.
There exists at least one ground state solution. Moreover, if
$U_1,U_2\in\cS$ satisfy $U_1\le U_2$ on $\mathbb R^N$, then
$U_1=U_2$.
\end{proposition}

\begin{proof}
By Lemma~\ref{lem:shooting-crossing}, $W(\cdot;q)$ reaches
$\theta^m$ at a finite radius $R(q)$ and $W_r(R(q);q)<0$. When $W(r;q)>\theta^m$,
\[
 r^{N-1}W_r(r;q)=-\int_0^r s^{N-1}h(W(s;q))\,ds<0.
\]
Thus the continuous dependence on $q$ shows that $R(q)$ and $\alpha(q)$ are continuous.

We next prove that $R(q)\alpha(q)\to\infty$ as $q\uparrow1$. Since $W(\cdot;1)\equiv1$,
the continuous dependence  gives
$R(q)\to\infty$ as $q\uparrow1$. Fix $\eta\in(\theta^m,1)$, let $r_\eta(q)$ be
the first radius where $W(r_\eta(q);q)=\eta$, then $r_\eta(q)\to\infty$ as $q\uparrow1$. Set
\[
 H(s)=\int_{\theta^m}^s h(\zeta)\,d\zeta.
\]
By \eqref{eq:W}, we have
\[
 \frac d{dr}\left(\frac12W_r^2+H(W)\right)
 =-\frac{N-1}{r}W_r^2\le0.
\]
Note that $W_r(0;q)=0$, thus $W^2_r(r_\eta(q);q)\leq 2(H(q)-H(\eta))$ is bounded as $q\uparrow1$. Given any
sequence $q_n\uparrow1$, we may therefore pass to a subsequence, denoted again by $q_n$,
such that $W_r(r_\eta(q_n);q_n)\to p_0\le0$. Define
$Z_n(s)=W(r_\eta(q_n)+s;q_n)$. Then
\[
 Z_n''+\frac{N-1}{r_\eta(q_n)+s}Z_n'+h(Z_n)=0.
\]
On every bounded $s$-interval, the coefficient of $Z_n'$ tends
uniformly to zero. Continuous
dependence therefore implies that $Z_n$ converges in $C^1$ on
compact $s$-interval to the solution $Z$ of
\[
 Z''+h(Z)=0,\qquad Z(0)=\eta,\qquad Z'(0)=p_0.
\]
Multiplying by $Z'$ and integrating give
\begin{equation}\label{x14}
  \frac12Z'^2+H(Z)=E_0:=\frac12p_0^2+H(\eta)>0.
\end{equation}
Since $p_0\leq0$ and $Z''(0)=-h(\eta)<0$, then $Z'<0$ for small positive $s$. Thus \eqref{x14} implies $Z'<0$ when $\theta^m<Z<\eta$. Then $Z$ reaches
$\theta^m$ at
\[
 s_* =\int_{\theta^m}^{\eta}
 \frac{d\zeta}{\sqrt{2(E_0-H(\zeta))}}<\infty.
\]
Indeed, the integrand is bounded near $\eta$ if $p_0<0$. If $p_0=0$, $E_0-H(\zeta)=h(\eta)(\eta-\zeta)+o(\eta-\zeta)$ as $\zeta\to\eta$, thus $s_*$
is still bounded. At $s=s_*$,
\[
 -Z'(s_*)=\sqrt{p_0^2+2H(\eta)}\ge\sqrt{2H(\eta)}>0.
\]
Then  $Z(s)$ remains positive on $[0,s_*+\delta]$ for small
$\delta>0$ and crosses $\theta^m$ only at $s_*$. The continuous
dependence and $Z'(s_*)<0$ show that the first
point where $Z_n=\theta^m$ converges to $s_*$ and $Z_n'(s_*)$ converges to $Z'(s_*)$. Since the sequence $q_n\uparrow1$ is arbitrary, we obtain
\[
 \liminf_{q\uparrow1}\alpha(q)\ge\sqrt{2H(\eta)}>0.
\]
It follows that $R(q)\alpha(q)\to\infty$ as $q\uparrow1$. By \eqref{eq:shooting-slope-bound}, we have $R(q)\alpha(q)\to0$ as $q\downarrow\theta$, then the intermediate value
theorem gives a value of $q$ satisfying \eqref{eq:matching}. Thus, there is a solution of \eqref{eq:ground-state-definition}, $\cS$ is nonempty.

To prove the second assertion, suppose that $U_1\le U_2$ and
$U_1\ne U_2$. Set $W_i=U_i^m$. The strong maximum principle
gives $W_1<W_2$. Let $R_i$ be the radius where $W_i=\theta^m$.
Then $R_1<R_2$ and when $r>R_2$,
\[
 W_i(r)=A_ir^{2-N},\qquad A_i=\theta^mR_i^{N-2},
\]
so $A_1<A_2$. We compare $W_2(|x|)$ with $W_1(|x-se_1|)$.
For $s$ in any bounded interval $I$,
\[
 W_1(|x-se_1|)-W_2(|x|)
 =(A_1-A_2)|x|^{2-N}+o(|x|^{2-N})<0
\]
uniformly with respect to $s$ in $I$ as $|x|\to\infty$.
The inequality holds for all $x$ when $s=0$. On the other hand, for  sufficiently large $s_1>0$, when $x=s_1e_1$,  $W_1(|x-se_1|)-W_2(|x|)=W_1(0)-W_2(s_1)>0$. Define
\[
 s_*:=\sup\{s\in[0,s_1]:W_1(|x-\sigma e_1|)<W_2(|x|)
 \text{ for all }x\in\R^N,\ 0\le\sigma<s\}.
\]
The uniform inequality for large $|x|$ and continuity on compact
sets give $s_*>0$ and $W_1(|x-s_*e_1|)\le W_2(|x|)$. By the definition of $s_*$, there is a $x_0>0$ such that $W_1(|x_0-s_*e_1|)= W_2(|x_0|)$. Thus the nonnegative
function
\[
 D(x)=W_2(|x|)-W_1(|x-s_*e_1|)
\]
solves $\Delta D+c(x)D=0$ with $c\in L^\infty$ and $D(x)$ vanishes
at an interior point. The local Harnack inequality gives $D\equiv0$.
This contradicts $W_1(0)<W_2(0)$ and  $A_1<A_2$.
Therefore $U_1=U_2$.
\end{proof}

\section{Zero number argument and general convergence}\label{sec:convergence}

In this section, we first prove a diminishing property for the number
of intersection points between two radial solutions of \eqref{eq:PDE}. We then use it to prove every solution of \eqref{eq:PDE} convergence
to a stationary solution. We regard each radial profile as a regular
radial function on $\mathbb R^N$. We assume \textup{(F1)} throughout
this section; \textup{(F3)} is needed only in
Theorem~\ref{thm:conv}.

The next proposition concerns intersection points  between  the solutions $u$ and a positive stationary
solution. Its proof uses the arguments in
\cite[Sections 3.1 and 3.3]{LouZhou2024}, together with the radial
zero number theorem of Chen and Pol\'a\v{c}ik \cite{ChenPolacik1996}.

\begin{proposition}\label{thm:radialzero}\label{lem:positive-stationary-zero}
 Let $I$ be an open time interval and let $L>0$. Let $u$ be a bounded,
nonnegative radial very weak solution of the equation in
\eqref{eq:PDE} on $B_L\times I$, continuous on
$\overline B_L\times I$ and classical wherever it is positive.
Let $V>0$ be a regular radial solution of
\[
 V_{rr}+\frac{N-1}{r}V_r+h(V)=0,
 \qquad h(s)=f(s^{1/m}),\qquad V_r(0)=0,
\]
on $[0,L]$. Suppose that $u^m(L,t)\ne V(L)$ for every $t\in I$
and define the number of intersection points between $u^m$ and $V$ by
\[
 \mathcal Z_L(t):=\#\{r\in[0,L]:u^m(r,t)=V(r)\}.
\]
Then $\mathcal Z_L(t)$ is finite for every $t\in I$ and is
nonincreasing with respect to $t$. If
\[
 u^m(r_*,t_*)=V(r_*),\qquad
 (u^m-V)_r(r_*,t_*)=0
\]
for some $r_*\in[0,L)$ and $t_*\in I$, then
$\mathcal Z_L(t_1)>\mathcal Z_L(t_2)$ for
$t_1,t_2\in I$ and $t_1<t_*<t_2$.
In particular, $u^m(0,t)-V(0)$ cannot vanish throughout a nonempty 
open time interval.

If $u$ solves \eqref{eq:PDE} with initial datum in $\mathscr X$
and $V$ is positive and bounded on $[0,\infty)$, the same
conclusions hold for the number of intersection points
on $[0,\infty)$ for all $t>0$.
\end{proposition}

\begin{proof}
 Write $Y=u^m$, $D=Y-V$ and
$a(y)=my^{(m-1)/m}$. In the positive set of $u$, the function $D$ satisfies
\begin{equation}\label{eq:stationary-difference-exact}
 D_t=a(Y)\left(D_{rr}+\frac{N-1}{r}D_r+cD\right),
 \qquad
 c(r,t)=\frac{h(Y(r,t))-h(V(r))}{Y(r,t)-V(r)}.
\end{equation}
When $Y=V$, set $c(r,t)=0$, thus $c(r,t)$ is bounded. With $\Delta_rD=D_{rr}+(N-1)D_r/r$, we have
\begin{equation}\label{eq:stationary-difference-frozen}
 D_t=a(V)\Delta_rD+c_0D,
 \qquad
 c_0=\frac{a(Y)-a(V)}{Y-V}\Delta_rD+a(Y)c.
\end{equation}
At points where $Y=V$, the quotient is understood as $a'(V)$.
Interior parabolic estimates show that $c_0$ is bounded on each
compact subset of the region where $Y,V>0$.

 On the region $B_L\backslash B_l\times I$ for any $l\in(0,L)$,  \eqref{eq:stationary-difference-frozen} satisfies the
coefficient conditions in \cite[Proposition 3.1]{LouZhou2024}, then we have all the assertion hold. To include the origin, we use \cite[Theorem 2.1, pp.~22--23]{ChenPolacik1996}. This theorem applies to nonzero radial generalized solutions of $v_t=\Delta v+c(x,t)v$ on a ball with homogeneous Dirichlet
boundary values, and requires only that $c$ be radial and bounded
measurable. It counts all zeros of the radial profile, including
the origin and the outer boundary, and  the number of zero is  strictly increases when $v$ has a multiple zero, also when that zero is at the origin.

Fix a closed bounded time interval $J$ contained in $I$.
Since $D(L,t)\ne0$ for $t\in J$, continuity allows us to decrease
$L$ slightly without removing any intersection point. We keep the
notation $L$ for the smaller radius. All the estimates below can
therefore be obtained from interior parabolic estimates. We construct a function that has the same zeros as $D$ and satisfies the conditions in \cite[Theorem 2.1, pp.~22--23]{ChenPolacik1996}.
Choose $0<\delta<\frac14\min_{[0,L]}V$. As in
\cite[Section 3.3]{LouZhou2024}, replace $D$ by
$Z=\Psi(D)$, where $\Psi\in C^2$ and $\Psi'\geq0$, agrees with the identity for
$D\ge-\delta/2$, equals $-\delta$ for $D\le-\delta$. The functions $D$ and $Z$ have the same zeros and agree near
each zero.  On the region $D\geq -\delta/2$, $Z=D$ and $Y>V-\delta/2>3\delta$,  by \eqref{eq:stationary-difference-frozen}, we have $Z$ is classical and satisfies
\begin{equation}\label{x19}
 Z_t=A(r)\Delta_r Z+C(r,t)Z,
 \qquad A(r)=a(V(r))>0,
\end{equation}
with $C$ bounded on $[0,L]\times J$. On the region $D\leq -\delta$, $Z$ is constant and satisfies \eqref{x19} with $C=0$.
On the region $D\in (-\delta,-\delta/2)$, we have $Y>V-\delta>3\delta$ and $|Z|$ has a positive lower bound, so $(Z_t-A\Delta_rZ)/Z$ is bounded.

To obtain a homogeneous Dirichlet condition, note first that on an annulus near $r=L$ for $t\in J$, $Z$
has a fixed nonzero sign. Extend $Z$ to a larger ball $B_{L+h}$ such that it is a time-independent
function $C^2$ with the same sign on $B_{L+h}\backslash B_{L}$ and vanishes at the new outer boundary $\partial B_{L+h}$
and is radial and harmonic near $\partial B_{L+h}$. Near the new outer boundary, $Z_t-A\Delta_rZ=0$. On the annulus $B_{L+h-\epsilon}\backslash B_{L}$, the absolute
value of the extended function has a positive lower bound. Hence $(Z_t-A\Delta_rZ)/Z$ is bounded.  This extension
adds exactly one zero, which is a fixed simple zero at that boundary.

We next change the radial variable so that the coefficient of the
second derivative becomes one. Define
\[
 \rho(r)=\int_0^r A(s)^{-1/2}\,ds,
 \qquad \widehat Z(\rho,t)=Z(r(\rho),t).
\]
The equation becomes
\[
 \widehat Z_t=\widehat Z_{\rho\rho}
 +\left(\frac{N-1}{\rho}+B(\rho)\right)\widehat Z_\rho
 +C(r(\rho),t)\widehat Z,
\]
where, with $r=r(\rho)$,
\[
 B(\rho)=\frac{(N-1)\sqrt{A(r)}}r
 -\frac{A'(r)}{2\sqrt{A(r)}}-\frac{N-1}{\rho}.
\]
Since $A\in C^2$ and $A'(0)=0$, Taylor's formula gives
$B(\rho)=O(\rho)$ and $B'(\rho)=O(1)$ as $\rho\downarrow0$. Then $B(\rho)$ and $B'(\rho)$ are bounded.
Set
\[
 \widehat Z(\rho,t)
 =\exp\left(-\frac12\int_0^\rho B(s)\,ds\right)v(\rho,t).
\]
Direct differentiation gives
\[
 v_t=v_{\rho\rho}+\frac{N-1}{\rho}v_\rho
       +\widetilde C(\rho,t)v,
 \qquad
 \widetilde C=C(r(\rho),t)-\frac{B'}2
             -\frac{N-1}{2\rho}B-\frac{B^2}4.
\]
Thus $\widetilde C$ is bounded also at the origin. The exponential
factor is positive, so the zeros and their multiplicities are
unchanged. Since $B(0)=0$, we have $v_\rho(0,t)=0$. At the new outer boundary, $v=0$. Rescaling space and time by
constants gives the same equation on the unit ball, thus by
\cite[Theorem 2.1]{ChenPolacik1996}, we have the assertion for $v$. Subtracting the one
fixed outer boundary zero gives the same assertions on $[0,L]$ for $u^m$ and $V$.
Since $J$ is arbitrary, they hold on all of $I$.
For a bounded open interval $I$, take $L>\max_{t\in I}b(t)$, then $u^m(L,t)-V(L)=-V(L)<0$ for $t\in I$, we have the assertion on $[0,\infty)$ for $t\in I$. Since $I$ is arbitrary, we have the assertion on $[0,\infty)$ for $t\in[0,\infty)$
\end{proof}

\begin{lemma}\label{lem:central-limit}
There exists $q_*\in[0,1]$ such that $u(0,t)$ converges to $q_*$ as $t\to\infty$.
\end{lemma}

\begin{proof}
Proposition~\ref{prop:basic} gives $b(t)\to\infty$ and
$\limsup_{t\to\infty}\|u(\cdot,t)\|_\infty\le1$.

Fix $q\in(0,1)$ and solve the regular radial stationary equation \eqref{eq:W} with
$V_q(0)=q^m$, $V_q'(0)=0$. For $q\le\theta$, we have $V_q\equiv q^m$. If $q>\theta$, the identity
\[
 r^{N-1}V_q'(r)=-\int_0^r s^{N-1}h(V_q(s))\,ds
\]
shows that $V_q$ is strictly decreasing when $V_q>\theta^m$.
If $V_q$ reaches $\theta^m$ at $r^*$, then $V_q'(r^*)<0$ and $V_q$ solves the radial harmonic equation for $r>r^*$
as long as $V_q$ remains positive. Therefore either $V_q$ is positive
 on $[0,\infty)$, or $V_q$ has a zero $S_q<\infty$
with $V_q'(S_q)<0$.

If $V_q(r)>0$ for $r\in[0,\infty)$, fix any $T>0$. By
Proposition~\ref{lem:positive-stationary-zero}, the zero number of
$u^m(\cdot,t)-V_q$ is finite at $t=T$ and nonincreasing for $t>T$.
When $u(0,t)=q$, $u^m(\cdot,t)-V_q$ has a multiple zero at the origin,
so that its zero number strictly decreases. Hence $u(0,t)=q$ can hold at only finitely many times for $t>T$.
If $V_q(S_q)=0$, choose $T_q$ sufficiently large such that $b(t)>S_q$ for all
$t\ge T_q$. Fix any $t_*>T_q$. The persistence of positivity of $u(S_q,t_*)$ and
continuity give a radius $L\in(0,S_q)$ and an open time interval
$J$ containing $t_*$, with $J\subset(T_q,\infty)$, such that
\[
 D_q(r,t):=u^m(r,t)-V_q(r)>0
 \qquad(L\le r\le S_q,\ t\in J).
\]
Note taht $V_q>0$ on $[0,L]$ and $D_q(L,t)>0$ on $J$,
thus Proposition~\ref{thm:radialzero} can be applied to $u$ and $V_q$ on the ball $B_L$, we have 
all zeros on $[0,S_q]$ lie in $[0,L)$. Their number is
finite and nonincreasing on $J$, and strictly decreases when time
passes through a value $\overline{t}$ such that $u(0,\overline{t})=q$.
The intervals $J$ cover $(T_q,\infty)$, so the same zero number properties hold for all $t>T_q$.
Since $D_q(r,t)$ has  only a finite number of zeros at $t=T_q$, there exists $T^*>T_q$ such that $u(0,t)\neq q$ for all $t>T^*$.

If the lower and upper limits of $u(0,t)$ are different as $t\to \infty$, we could
choose $q\in(0,1)$ strictly between them. Then the continuity  gives a
sequence $t_n\to\infty$ such that $u(0,t_n)=q$, a contradiction.
The lemma follows.
\end{proof}

\begin{proposition}\label{prop:quasi}
For $u_0\in\mathscr X$, the $\omega$-limit set of $u$ is nonempty,
compact and connected in $C_{\loc}([0,\infty))$. Every element of
$\omega(u)$ is a bounded nonnegative radial stationary solution on
$\mathbb R^N$. Moreover, $\omega(u)$ contains exactly one function $U$  and $u(\cdot,t)$ converges to  $U$
locally uniformly in $\mathbb R^N$ as $t\to\infty$. For any sequence $t_n\to\infty$, the functions $u(x,t_n+s)$ also
converge in $C^{2,1}$ on each compact subset of space and time where $U>0$.
\end{proposition}

\begin{proof}
By the uniform H\"older estimates and Cantor's diagonal argument,
for any sequence $t_n\to\infty$, we can choose a subsequence,
denoted again by $t_n$, such that
\[
 u(r,t_n+s)\longrightarrow w(r,s)
\]
locally uniformly for $(r,s)\in[0,\infty) \times \mathbb R$ as
$n\to\infty$. The local uniform convergence of $u$, $u^m$, and
$f(u)$ allows us to pass to the limit in the weak formulation \eqref{eq:very-weak}.
Thus $w$ is a bounded radial very weak solution defined for all
$s\in\mathbb R$. This also proves that $\omega(u)$ is nonempty
and compact. For each $T>0$, the set $\{u(\cdot,t):t\ge T\}$ is
connected and its closure is compact and connected. These closures
decrease as $T$ increases, and their intersection is $\omega(u)$.
Hence $\omega(u)$ is connected.
On any compact subset of the region where $w>0$, both $w$ and
$u(r,t_n+s)$ have a common positive lower bound for large $n$.
Interior parabolic estimates then give convergence in $C^{2,1}$.

Lemma~\ref{lem:central-limit} gives
\begin{equation}\label{eq:entire-center-fixed}
 w(0,s)=q_*\qquad(s\in\mathbb R).
\end{equation}
Suppose first that $q_*=0$. If $w(x_0,s_0)>0$ for some $x_0\in\mathbb R$,
choose a sufficiently small ZKB solution $B$, centered at $x_0$,
such that $B(\cdot,t_0)$ lies below $w(\cdot,s_0)$ for some $t_0>0$ and  the maximum value of $B(\cdot,t_0)$
is  less than $\theta$. Thus there exists a $s^*>0$ such that $B(0,t_0+s^*)>0$, by comparison
principle we have $w(0,s_0+s^*)>0$, contradicting
\eqref{eq:entire-center-fixed}. Thus $w\equiv0$.

Suppose $q_*>0$. The uniform H\"older bound and \eqref{eq:entire-center-fixed}
give $\rho,c_0>0$, independent of $s$, such that
\[
 w(r,s)\ge c_0\qquad(s\in\mathbb R,\ 0\le r\le\rho).
\]
Choose a ZKB solution $B$ such that for some $t_0>0$, $\operatorname{supp}B(\cdot,t_0)\subset B_\rho$ and the maximum value of $B(\cdot,t_0)$
is  less  than $\min\{c_0,\theta\}$.
For every fixed $R>0$, choose $T_R$ such that $B(\cdot,T_R+t_0)$ is strictly
positive on $\overline{B_R}$. At the time $s-T_R$, the same profile
$B(\cdot,t_0)$ lies below $w(\cdot,s-T_R)$. Comparison gives
\begin{equation}\label{eq:entire-positive-on-balls}
 w(r,s)\ge B(r,T_R)\ge c_R>0
 \qquad(s\in\mathbb R,\ 0\le r\le R),
\end{equation}
where $c_R$ is independent of $s$.

Let $V_*$ be the regular radial stationary solution with
$V_*(0)=q_*^m$ and $V_*'(0)=0$. Let $[0,S_*)$ be the largest
interval on which $V_*$ is positive. If $q_*=1$, then $V_*\equiv1$.
We will show that
\[
 w^m(r,s)=V_*(r)\qquad(s\in\mathbb R,\ 0\le r<S_*).
\]
Suppose, to the contrary, that $D=w^m-V_*$ is nonzero at some
point $(r_1,s_1)$ with $0<r_1<S_*$. By continuity, $D(r_1,s)$ has a fixed nonzero sign
on an open  time interval $J$  containing $s_1$. Note that $V_*>0$ on $[0,r_1]$, 
Proposition~\ref{thm:radialzero} therefore can be applied to to $w$ and $V_*$
on $[0,r_1]\times J$. On the other hand,
\[
 D(0,s)=D_r(0,s)=0\qquad(s\in J).
\]
Let the  number of zeros of $D(\cdot,s_-)$ be $k$, and choose
$s_+\in J$ with $s_+>s_-$. Take $k+1$ disjoint closed intervals
of positive length in $(s_-,s_+)$. On each interval, the multiple
zero at the origin forces a strict decrease of the zero number.
This contradict the fact that $D(\cdot,s_-)$ has $k$ zeros at time $s_-$. The claimed equality follows.

If $S_*<\infty$, continuity would give $w(S_*,s)=0$, contradicting
\eqref{eq:entire-positive-on-balls} with $R=S_*$. Hence $S_*=\infty$,
and $w(r,s)=V_*(r)^{1/m}$ is independent of $s$. The weak formulation \eqref{eq:very-weak}
then shows that $w$ is stationary on $\mathbb R^N$. Since the sequence
$t_n$ was arbitrary, every member of $\omega(u)$ is stationary.

If $q_*=0$, the only possible limit is zero. If $q_*>0$,
uniqueness for the regular radial stationary equation with  central value $q_*^m$ shows that all limits are the same. Thus $\omega(u)$
contains exactly one function $U$. If $u(\cdot,t)$ did not converge to
$U$ locally uniformly, there would be a ball $B_L$ and a sequence
$t_n\to\infty$ such that $\max_{r\in[0,L]}|u(r,t_n)-U|>c>0$. The uniform
H\"older estimates would give a subsequence converging to another
member of $\omega(u)$, a contradiction. This proves convergence.
\end{proof}

\begin{lemma}
\label{lem:excludeplateau}
No constant $q\in(0,\theta)$ belongs to the $\omega$-limit set.
If $N\ge3$, no nonconstant radial stationary solution with limit
in $(0,\theta)$ as $r\to\infty$ belongs to the $\omega$-limit set.
\end{lemma}

\begin{proof}
Suppose that $u(\cdot,t_n)$ converges to a constant $q\in(0,\theta)$
locally uniformly in $\mathbb R^N$ as $t_n\to\infty$. Choose
$R_0$ larger than the initial support radius. For sufficiently large $n$,
$u(\cdot,t_n)<\theta$ on $[0,R_0]$. Monotonicity outside the
initial support gives $u(\cdot,t_n)<\theta$ on $\mathbb R^N$.
Comparison with the constant $\theta$ shows that $u\le\theta$,
and hence $f(u)=0$ for all $t>t_n$.  The finite propagation speed allow us to choose a ZKB solution $B(\cdot,t_0)>u(\cdot,t_n)$.
Since $\max B(\cdot,t_0+t)\to 0$  as $t\to\infty$, comparison shows
that $u(\cdot,t)$ converges to zero uniformly in $\mathbb R^N$ as
$t\to\infty$, which contradicts $q>0$.

Now let $N\ge3$, and suppose that $Q\in\omega(u)$ is a nonconstant
radial stationary solution satisfying
\[
 \lim_{r\to\infty}Q(r)=q_\infty\in(0,\theta).
\]
By Proposition~\ref{prop:quasi}, $u(\cdot,t)$ converges to $Q$
locally uniformly in $\mathbb R^N$ as $t\to\infty$.
Choose $R$ larger than the initial support radius such that
\[
 q_\infty<Q(R)<\theta,
\]
and fix $B$ with $Q(R)<B<\theta$.
There exists $T>0$ such that $u(R,t)\le B$ for every $t\ge T$.
By monotonicity outside the initial support,
\[
 0\le u(r,t)\le B<\theta\qquad(r\ge R,\ t\ge T).
\]
Thus $u$ solves the porous medium equation without reaction for $r>R$ and $t\ge T$.

Put $\nu=(N-2)/m$ and, for $a>0$, define
\[
 H_a(r)=a\left(\frac Rr\right)^\nu,\qquad r\ge R.
\]
$H_a^m$ is harmonic for $r>R$. Since $u(\cdot,T)$ is bounded with
compact support, choose $A\ge B$ such that
\[
 u(r,T)\le H_A(r)\qquad(r\ge R).
\]
If $b(T)>R$, it suffices to take
$A\ge\max\{B,\max_{R\le r\le b(T)}u(r,T)(r/R)^\nu\}$;
otherwise any $A\ge B$ is sufficient.
For $L>R$, let $z_L$ be the radial solution of
\[
 \begin{cases}
 (z_L)_s=\Delta z_L^m,&R<|x|<L,\quad s>0,\\
 z_L=B,&|x|=R,\quad s>0,\\
 z_L=H_A(L),&|x|=L,\quad s>0,\\
 z_L(x,0)=H_A(|x|),&R<|x|<L.
 \end{cases}
\]
The initial condition is understood in $L^1$ on the open annulus,
so compatibility at $r=R$, $s=0$ is unnecessary. Comparison with
the stationary solutions $H_B$ and $H_A$ gives
\begin{equation}\label{eq:annular-harmonic-bounds}
 H_B(r)\le z_L(r,s)\le H_A(r)\qquad(R<r<L,\ s>0).
\end{equation}
Since $z_L(\cdot,s_0)\le H_A=z_L(\cdot,0)$, comparison with time
translates gives
\[
 z_L(r,s+s_0)\le z_L(r,s)\qquad(s,s_0>0).
\]
If $L_2>L_1$, then
$z_{L_2}(L_1,s)\le H_A(L_1)=z_{L_1}(L_1,s)$, while the initial
data and the inner boundary values agree. Therefore
\[
 z_{L_2}(r,s)\le z_{L_1}(r,s)
 \qquad(R<r<L_1,\ s>0).
\]
Hence $z(r,s)=\lim_{L\to\infty}z_L(r,s)$ exists, satisfies
$H_B\le z\le H_A$, and is nonincreasing in time. On each fixed bounded annulus, $H_B$ has a positive lower bound,
so the equations are uniformly parabolic there. Interior and boundary
parabolic estimates, on closed bounded time intervals contained in
$(0,\infty)$, allow us to pass to the limit in the equation and at
$r=R$. Hence $z_s=\Delta z^m$ for $r>R$ and $z(R,s)=B$.
The estimates near $r=R$ are also uniform for $s\ge1$.

Fix $S>0$. The finite propagation  speed gives $L_S>R$ such that
$u(r,T+s)=0$ for $r\ge L_S$ and $0\le s\le S$.
Since $u(\cdot,T)\le H_A$ for $r\geq R$,
$u(R,T+s)\le B$ for $s>0$ and $u(L,T+s)=0$ for  $L>L_S$, $0\le s\le S$. Comparison on the annulus
therefore gives
\[
 u(r,T+s)\le z_L(r,s)\qquad(R<r<L,\ 0<s\le S).
\]
Letting $L\to\infty$ and recalling that $S>0$ was arbitrary, we obtain
\begin{equation}\label{eq:exterior-comparison-decay}
 u(r,T+s)\le z(r,s)\qquad(r>R,\ s>0).
\end{equation}
Since $z$ is bounded and nonincreasing in $s$, it has a limit
$Z(r)$ as $s\to\infty$. On each bounded annulus, parabolic
estimates allow us to choose convergent subsequences of
$z(r,s_n+s)$ in $C^{2,1}$ as $s_n\to\infty$. Monotonicity of $z$ in $s$ shows that the
limit is $Z(r)$ for every $s$ and $\{s_n\}$. Passing to the limit in the equation
gives $\Delta Z^m=0$. The boundary estimates give $Z(R)=B$,
and $H_B\le Z\le H_A$ gives
$Z(r)\to0$ as $r\to\infty$. Since a radial harmonic function has
the form $c_1+c_2r^{2-N}$, thus
\[
 Z(r)^m=B^m\left(\frac Rr\right)^{N-2},
 \qquad Z(r)=H_B(r).
\]
Letting $s\to\infty$ in \eqref{eq:exterior-comparison-decay}, we obtain
\[
 Q(r)\le H_B(r)\qquad(r>R).
\]
Thus $\lim_{r\to\infty}Q(r)=q_{\infty}\leq \lim_{r\to\infty}H_{b}(r)=0$, which  is a contradiction.
\end{proof}

\begin{theorem}\label{thm:conv}
Assume \textup{(F1)} holds, and \textup{(F3)} also  holds  when $N\ge3$.
\begin{enumerate}[label=(\roman*)]
\item If $N=2$, every solution with an initial datum in $\mathscr X$
converges to one of the constants $0$, $\theta$, and $1$ locally
uniformly in $\mathbb R^2$ as $t\to\infty$.
\item If $N\ge3$, every solution with an initial datum in $\mathscr X$
converges to one of the constants $0$, $\theta$, and $1$, or to a
member of $\mathcal S$, locally uniformly in $\mathbb R^N$ as
$t\to\infty$.
\end{enumerate}
\end{theorem}

\begin{proof}
By Proposition~\ref{prop:quasi}, $u(\cdot,t)$ converges to a radial
stationary solution locally uniformly in $\mathbb R^N$ as $t\to\infty$. By (F1) and Lemma~\ref{lem:excludeplateau},
any constant limit belongs to $\{0,\theta,1\}$. For $N=2$,
Lemma~\ref{lem:noGS2} excludes nonconstant limits. For $N\ge3$,
every nonconstant limit satisfies $U(0)\in(\theta,1)$ by
Lemma~\ref{lem:stationarygeometry}. 
Lemma~\ref{lem:shooting-crossing} shows that $U^m$ reaches $\theta^m$
at a finite radius $R$ and $U'(r)<0$ for $r>R$. From the proof of Proposition \ref{prop:quasi}, we have $U>0$ on $\mathbb{R}$ and
Lemma~\ref{lem:excludeplateau} shows that $\lim_{r\to\infty}U=0$.
Thus $U$ belongs to $\mathcal S$.
\end{proof}

\begin{proposition}\label{prop:uniform}
Suppose that $u(\cdot,t)$ converges to zero locally uniformly in
$\mathbb R^N$ as $t\to\infty$. Then $u(\cdot,t)$ converges to zero
uniformly in $\mathbb R^N$ as $t\to\infty$. If $N\ge3$ and
$u(\cdot,t)$ converges to a solution $U\in\cS$ locally uniformly in
$\mathbb R^N$ as $t\to\infty$, then $u(\cdot,t)$ converges to $U$
uniformly in $\mathbb R^N$ as $t\to\infty$.
\end{proposition}

\begin{proof}
For $R>R_0$, monotonicity outside the initial support gives
$\sup_{r\ge R}u(r,t)\le u(R,t)$.
If $u(\cdot,t)$ converges to zero locally uniformly in
$\mathbb R^N$ as $t\to\infty$, this inequality and uniform
convergence on $[0,R]$ show that $u(\cdot,t)$ converges to zero
uniformly in $\mathbb R^N$ as $t\to\infty$.

Since $U\in\cS$, for any $\epsilon>0$, there exists $R>R_0$ such that $U(r)<\epsilon$ for $r\geq R$. If $u(\cdot,t)$ converges to $U\in\cS$ locally uniformly in
$\mathbb R^N$ as $t\to\infty$, there exists $T_R>0$ such that for $t>T_R$
\[
 \sup_{r\le R}|u(r,t)-U(r)|\le \epsilon.
\]
 Monotonicity outside the initial support gives $u(r,t)\leq u(R,t)\leq U(R)+\epsilon\leq2\epsilon$ for $t>T_R$ and $r>R$,
Thus $u(\cdot,t)$
converges to $U$ uniformly in $\mathbb R^N$ as $t\to\infty$.
\end{proof}

\section{Exclusion of \texorpdfstring{$\theta$}{theta} from the \texorpdfstring{$\omega$}{omega}-limit set}
\label{sec:notheta-normalized}

Throughout this section, we assume $N\ge3$ and \textup{(F1)}, and
let $u$ be the solution of \eqref{eq:PDE} with $u_0\in\mathscr X$.
We first exclude convergence to $\theta$ when $u(\cdot,t)-\theta$
has exactly one zero at some positive time.

\begin{lemma}\label{lem:one-intersection}
Assume in addition \textup{(F2)} and \textup{(F3)}.
If, at some time $T>0$, $u(\cdot,T)-\theta$
has exactly one zero $r_*>0$ and
\begin{equation}\label{x25}
 u(r,T)>\theta\quad(0\le r<r_*),\qquad
 u(r,T)<\theta\quad(r>r_*),\qquad u_r(r_*,T)<0,
\end{equation}
then $u(\cdot,t)$ cannot converge to $\theta$ locally uniformly in
$\mathbb R^N$ as $t\to\infty$.
\end{lemma}

\begin{proof}
For $q>\theta$ sufficiently close to $\theta$,
Lemma~\ref{lem:shooting-crossing} and \eqref{eq:exterior} give a
positive stationary solution
$V_q=W(\cdot;q)^{1/m}$ satisfying
\begin{equation}\label{eq:comparison-stationary-tail}
 V_q(0)=q,\qquad
 \lim_{r\to\infty}V^m_q(r)
 =\theta^m-\frac{R(q)\alpha(q)}{N-2}\in(0,\theta^m).
\end{equation}
By the continuous dependence on the central value in
\eqref{eq:radial-integral}, $V_q\to\theta$ in $C^1([0,B])$ for
every $B>0$ as $q\downarrow\theta$.

Choose $B>b(T)$ and a closed interval $I=[l,g]\subset(0,b(T))$
containing $r_*$ in its interior such that $u_r(\cdot,T)<0$ on $I$.
For $q>\theta$ sufficiently close to $\theta$, since $V_q\to\theta$ in $C^1([0,B])$, together with \eqref{x25}, we have $D:=u(\cdot,T)-V_q$ has  exactly one simple zero in $I$ and $D>0$ on $[0,l]$, $D<0$ on $[g,B]$. For $r\ge B$, we have
$u(r,T)-V_q(r)=-V_q(r)<0$, thus $D$ has exactly one zero on $[0,\infty)$. In particular, $u(0,T)>q$.

Suppose that $u(\cdot,t)$ converges to $\theta$ locally uniformly in
$\mathbb R^N$ as $t\to\infty$. Then there exists $T_1>T$ such that $u(0,T_1)=q$.
At this time, $u^m-V^m_q(r)$ has a multiple zero at the origin.
Since $u^m(\cdot,t)-V^m_q(r)$ has exactly one zero for $t\in[T,T_1)$,
Proposition~\ref{lem:positive-stationary-zero} implies that $u^m(\cdot,t)-V^m_q(r)$ has no zero point for every $t>T_1$.
Since $u^m(r,T_1)-V^m_q(r)<0$ for $r>b(T_1)$, we have
$u(r,t)<V_q(r)$ for all $r\ge0$ and $t>T_1$.
Choose  $r>0$  such that $V_q(r)<\theta$. Letting $t\to\infty$, we obtain
$\theta\le V_q(r)<\theta$, a contradiction.
\end{proof}

We next choose a sequence of times and divide $u-\theta$ by its
maximum absolute value on the initial support ball. We will obtain a
nonzero limit satisfying a uniformly parabolic equation.

\begin{lemma}\label{lem:normalized-window-limit}
Suppose that $u(\cdot,t)$ converges to $\theta$ locally uniformly in
$\mathbb R^N$ as $t\to\infty$, and put
$H(t)=\max_{r\le R_0}|u(r,t)-\theta|$. There are times
$\tau_n\to\infty$ with $H(\tau_n)>0$ such that
\[
 \eta_n(x,s):=\frac{u(x,\tau_n+s)-\theta}{H(\tau_n)}
 \longrightarrow\eta(x,s)
\]
in $C^{2,1}$ on compact subsets of $\mathbb R^N\times[0,\infty)$,
where
\[
 \eta_s=d\Delta\eta+\kappa\eta_+,\qquad
 d=m\theta^{m-1},\quad\kappa=f'_+(\theta),\qquad
 \max_{|x|\le R_0}|\eta(x,0)|=1.
\]
Moreover, $\eta$ is bounded on $\mathbb R^N\times[0,S]$ for every
$S>0$, and
$\eta(r,s)\ge0$, $\eta_r(r,s)\le0$ for $r\ge R_0$ and $s\ge0$.
\end{lemma}

\begin{proof}
Put $M=\max\{1,\|u_0\|_\infty\}$, and let $K>0$ be a Lipschitz
constant of $f$ on $[0,M]$. By \eqref{eq:outermono}, we have
\[
 \|u(\cdot,t)\|_\infty=\max_{r\le R_0}u(r,t)\to\theta,
 \qquad H(t)\to0.
\]
Hence $u(x,t)<1$ for every $x\in\mathbb R^N$ and all sufficiently large $t$.

\smallskip\noindent
\emph{Step 1.} We first choose times $t_n\to\infty$, radii
$\rho_n\to\infty$, and numbers $\alpha_n,T_n>0$ such that
\begin{equation}\label{eq:normalization-positive-annuli}
 u(r,t_n+s)\ge\theta+\alpha_n
 \quad(R_0\le r\le\rho_n,\ 0\le s\le T_n),\qquad
 T_n-\ln\frac M{\alpha_n}\to\infty.
\end{equation}
Define
\[
 M_1(t)=\int_0^\infty r u(r,t)\,dr,\qquad
 Q(t)=\int_0^\infty r f(u(r,t))\,dr.
\]
By the finite propagation speed and the continuity of $u$, these integrals
are finite and continuous on every bounded time interval. We claim that
\begin{equation}\label{eq:heat-newtonian-moment}
 M_1'(t)=-(N-2)u^m(0,t)+Q(t).
\end{equation}
To prove this identity, fix a bounded time interval and set
$\varphi(x,t)=\psi_\varepsilon(x)C(x,t)=(|x|^2+\varepsilon^2)^{(2-N)/2}C(x,t)$ with $C(x,t)$ is a smooth  cutoff function, in the very weak formulation \eqref{eq:very-weak}. We choose
the cutoff function $C(x,t)$ equal to one on the supports of $u$ throughout this time
interval. A direct calculation gives
\[
 \Delta\psi_\varepsilon
 =-N(N-2)\varepsilon^2(|x|^2+\varepsilon^2)^{-(N+2)/2}
 \longrightarrow-(N-2)|\mathbb S^{N-1}|\delta_0
\]
in the sense of distributions as $\varepsilon\downarrow0$.
The functions $\Delta\psi_\varepsilon$ are negative and their
integrals equal $-(N-2)|\mathbb S^{N-1}|$. We pass to the limit in
the very weak formulation \eqref{eq:very-weak} and divide by $|\mathbb S^{N-1}|$.
By writing \eqref{eq:very-weak} in radial coordinates, we obtain  the time-integrated form of \eqref{eq:heat-newtonian-moment}. Since the right-hand side of \eqref{eq:heat-newtonian-moment} is
continuous, we have \eqref{eq:heat-newtonian-moment} holds at every positive time.

Since $\int_0^R r u(r,t)\,dr\to\theta R^2/2$ for every $R$,
$M_1(t)\to\infty$ as $t\to\infty$. Therefore $M_1$ attains its minimum on
$[n,\infty)$ at some $t_n\ge n$, then $M_1'(t_n)\ge0$. Consequently, for all sufficiently large $n$, we have
\[
 Q(t_n)\ge(N-2)u^m(0,t_n)\ge c_0>0.
\]
Set $g_n=(u(\cdot,t_n)-\theta)_+$ and
$\varepsilon_n=\|g_n\|_\infty$. Then $\varepsilon_n$ tends to zero
as $n\to\infty$.  Note that $f(u(r,t_n))=0$ if and only if $g_n=(u(r,t_n)-\theta)_+=0$ and when $f(u(r,t_n))>0$, $f(u(r,t_n))\leq K(u(r,t_n)-\theta)=Kg_n(r)$, we obtain
\begin{equation}\label{eq:heat-positive-height-integral}
 \int_0^\infty K r g_n(r)dr\geq Q(t_n)\,\ge c_0.
\end{equation}
Put $\Lambda=mM^{m-1}$, $A=4NM$, and choose $\delta>0$ so that
$(64N\Lambda\delta^2)^{-1}\ge8$. There exists $N_0>0$ such that for every  $n>N_0$, there is
$R_n\ge4R_0$ such that
\[
 a_n:=g_n(R_n)>Ae^{-\delta R_n}.
\]
Otherwise, there is a subsequence along which
\[
 \int_0^\infty r g_n(r)\,dr
 \le8R_0^2\varepsilon_n+
 \int_{4R_0}^\infty r\min\{\varepsilon_n,Ae^{-\delta r}\}\,dr
 \longrightarrow0.
\]
This contradicts
\eqref{eq:heat-positive-height-integral}. Since $a_n\le\varepsilon_n$ and $\varepsilon_n\to0 $, we have
$R_n\to\infty$. Set
\[
 L_n=\ln(A/a_n),\qquad
 T_n=\frac{R_n^2}{64N\Lambda L_n},
\]
we have $L_n\to\infty$, $L_n<\delta R_n$, and $T_n\ge8L_n$.

Put $x_n=(R_n/2)e_1$, $d_n=R_n/(4\sqrt N)$, and
$b_n^0=\theta+a_n$. Every point of $x_n+(-d_n,d_n)^N$ has radius
between $R_n/4$ and $3R_n/4$, so $u(\cdot,t_n)\ge b_n^0$ on this
cube by \eqref{eq:outermono}.

For $k_n=d_n/(2\Lambda T_n)$, define
\[
 P_n(y,s)=b_n^0\sum_{i=1}^N
 \left[e^{-k_n(d_n-y_i)+\Lambda k_n^2s}
       +e^{-k_n(d_n+y_i)+\Lambda k_n^2s}\right],
 \qquad z_n=(b_n^0-P_n)_+,
\]
and extend $z_n$ by zero outside $(-d_n,d_n)^N$.
Since $P_n>b_n^0$ on the boundary, $z_n$ vanishes near it.
On $\{z_n>0\}$, $(P_{n})_s=\Lambda\Delta P_n$ and
$\Delta P_n=k_n^2P_n\ge0$, for sufficiently large $n$, $z_n\le b_n^0<1$, then $f(z_n)\ge0$, hence
\[
 (z_n)_s-\Delta z_n^m-f(z_n)
 =\bigl(mz_n^{m-1}-\Lambda\bigr)\Delta P_n
   -m(m-1)z_n^{m-2}|\nabla P_n|^2-f(z_n)\le0.
\]
 Since $z_n$ is locally
Lipschitz and $\nabla z_n^m\to0$ at the boundary of $\{z_n>0\}$. We have the inequality also holds in the sense of distributions.
 Thus  $z_n$ is a very weak subsolution of \eqref{eq:PDE}. Together with $u(\cdot,t_n)\ge b_n^0$ on $x_n+(-d_n,d_n)^N$.
By comparison, we have
$u(x_n+y,t_n+s)\ge z_n(y,s)$. For $0\le s\le T_n$,
\[
 P_n(0,s)\le2Nb_n^0e^{-d_n^2/(4\Lambda T_n)}
 =2Nb_n^0e^{-L_n}\le a_n/2.
\]
It follows that $u(R_n/2,t_n+s)\ge\theta+a_n/2$ for
$0\le s\le T_n$. Set $\rho_n=R_n/2$ and $\alpha_n=a_n/2$.
The monotonicity in \eqref{eq:outermono} gives the first inequality
in \eqref{eq:normalization-positive-annuli}. Moreover,
\[
 T_n-\ln(M/\alpha_n)
 \ge8L_n-\ln(2M/a_n)\ge7L_n\to\infty.
\]

\smallskip\noindent
\emph{Step 2.} We choose $\tau_n\in[t_n,t_n+T_n]$ and prove that $\{\eta_n\}$ has a convergent subsequence. By
\eqref{eq:outermono} and \eqref{eq:normalization-positive-annuli},
\begin{equation}\label{eq:normalization-window-bounds}
 \alpha_n\le H(t)\le M,\qquad
 |u(r,t)-\theta|\le H(t)
 \quad(0\le r\le\rho_n,\ t_n\le t\le t_n+T_n).
\end{equation}
Since $H(t)\to0$, we may assume, after omitting finitely many
terms, that $H(t)<\min\{\theta,1-\theta\}$ for $t\ge t_n$.
If $t,t+\tau\in[t_n,t_n+T_n]$ and $\tau\ge0$, then
\begin{equation}\label{eq:normalization-forward-control}
 H(t+\tau)\le e^{K\tau}H(t).
\end{equation}
Note that the solution of $q'=f(q)$, $q(0)=\theta+H(t)$ is a supersolution of \eqref{eq:PDE} and $u(\cdot,t)\leq \theta+H(t)$. By comparison, we have
$u(x,t+\tau)\le q(\tau)\le\theta+e^{K\tau}H(t)$. The constant $\theta-H(t)$ is a subsolution of \eqref{eq:PDE}. For $s\in[t,t+\tau]$, $u(\rho_n,s)\geq\theta+\alpha_n\geq\theta-H(t)$, together with $u(\cdot,t)>\theta-H(t)$ on $B_{\rho_n}$,  by comparison, we have $u(x,t+\tau)\ge\theta-H(t)$ on $B_{\rho_n}$. These two inequalities prove
\eqref{eq:normalization-forward-control}.

Set $J_n=\lceil\ln(M/\alpha_n)\rceil+1$, where $\lceil a\rceil$ is  the smallest integer greater than or equal to $a$. By
\eqref{eq:normalization-positive-annuli}, $T_n-J_n\to\infty$. We claim that
there exists $j_n\in\{1,\ldots,J_n\}$ such that
\[
 H(t_n+j_n-1)\le eH(t_n+j_n).
\]
Otherwise, for every $j\in\{1,\ldots,J_n\}$, $H(t_n+j-1)> eH(t_n+j)$, then $H(t_n+J_n)<e^{-1}H(t_n+J_n-1)<e^{-J_n}M<\alpha_n$, contradicting
\eqref{eq:normalization-window-bounds}. Put $\tau_n=t_n+j_n$.
For each fixed $S>0$, $[\tau_n-1,\tau_n+S+1]$ lies in
$[t_n,t_n+T_n]$ for large $n$, and
\eqref{eq:normalization-forward-control} gives
\begin{equation}\label{eq:normalization-past-bound}
 \frac{H(\tau_n+s)}{H(\tau_n)}
 \le e^{1+K(S+2)}\qquad(-1\le s\le S+1).
\end{equation}
Since $\rho_n\to\infty$,
\eqref{eq:normalization-window-bounds} and \eqref{eq:normalization-past-bound} give a bound for $\eta_n$
on every fixed ball $B_R$ and for $-1\le s\le S+1$, independent of $n$.

To obtain derivative estimates, put $H_n=H(\tau_n)$ and
$Z_n=(u(x,\tau_n+s)^m-\theta^m)/H_n$. The functions $Z_n$ are
uniformly bounded on $B_R$ for $-1\le s\le S+1$ and satisfy
\[
 \begin{gathered}
 (Z_n)_s=A_n(x,s)\Delta Z_n+A_n(x,s)F_n(Z_n),\\
 A_n(x,s)=m u(x,\tau_n+s)^{m-1},\qquad
 F_n(z)=\frac{f((\theta^m+H_nz)^{1/m})}{H_n}.
 \end{gathered}
\]
On every bounded set $\Omega\subset \mathbb R^N\times[0,\infty)$, the
convergence $u(\cdot,t)\to\theta$ shows that  when $n$ is sufficiently large, $m\theta^{m-1}/2$ is a positive lower bound for
$A_n$ and the local H\"older estimates for $u$
also give uniform H\"older bounds for $A_n$ on $\Omega$. By \textup{(F1)}, we have $F_n(0)=0$ and $F_n$ is locally Lipschitz on $[0,\infty)$. Thus for sufficiently large $n$, the Schauder estimates shows that $Z_n\in C^{2+\beta,1+\beta/2}(\Omega')$ for some
$\beta\in(0,1)$, where $\Omega'=\overline{B}_{R-\epsilon}\times[-1/2,S] \subset B_R\times[-1,S+1]$.

By a diagonal argument,  there exist a subsequence of $\{Z_n\}$,  denoted again by $\{Z_n\}$, converges to $Z$ in
$C^{2,1}$ on compact subsets of $\mathbb R^N\times[0,\infty)$. The sequence of functions
\[
 f_n(z)=\frac{(\theta^m+H_nz)^{1/m}-\theta}{H_n}
\]
converges uniformly to $z/d$ in $C^2$ on any bounded intervals $I\in \mathbb{R}$ as $n\to\infty$,
thus $\eta_n$ also converges in $C^{2,1}$ on compact subsets of $\mathbb R^N\times[0,\infty)$ to a
function $\eta$. Since $A_n\to d$ and
$F_n(z)\to\kappa z_+/d$ locally uniformly, passing to the limit
in the equation shows that $\eta$ satisfies the stated equation.
The convergence on $\overline B_{R_0}$ at $s=0$ gives
$\max_{|x|\le R_0}|\eta(x,0)|=1$.
Since the bound in \eqref{eq:normalization-past-bound} does not depend
on $B_R$, we have $\eta$ is bounded on
$\mathbb R^N\times[0,S]$ for every $S>0$. Finally,
  \eqref{eq:outermono} and \eqref{eq:normalization-positive-annuli}
imply $\eta(r,s)\ge0$ and $\eta_r(r,s)\le0$ for $r\ge R_0$
and $s\ge0$.
\end{proof}

\begin{lemma}\label{lem:heat-mass}
Let $N\ge1$, $d>0$, $\kappa\ge0$ and $R_0>0$. Let $\eta$ be a
nonzero radial solution of
\begin{equation}\label{eq:heat-positive-part}
 \eta_t=d\Delta\eta+\kappa\eta_+
 \quad\hbox{in }\mathbb R^N\times(0,\infty),\qquad
 \eta_+=\max\{\eta,0\},
\end{equation}
continuous up to $t=0$, classical for $t>0$ and bounded on
$\mathbb R^N\times[0,T]$ for every $T<\infty$.
If
\begin{equation}\label{eq:heat-exterior-sign}
 \eta(r,t)\ge0,\qquad\eta_r(r,t)\le0\quad(r\ge R_0,\ t>0),
\end{equation}
then, for every $L>0$, there is $T_L>0$ such that
$\eta(x,t)>0$ for $|x|\le L$ and $t\ge T_L$.
\end{lemma}

\begin{proof}
Fix $R>R_0$ and put $\Omega_R=\mathbb R^N\setminus\overline B_R$.
We first show that $\eta$ is positive at some point of
$\Omega_R\times(0,\infty)$. Otherwise,
\eqref{eq:heat-exterior-sign} implies that $\eta=0$  on $\Omega_R\times(0,\infty)$. Since $\eta\not\equiv0$, there are
$r_1\in(0,R)$ and an open time interval $J\subset(0,\infty)$ such that  $\eta(r_1,t)\ne0$ for $t\in J$. Choose $R_1>R$. On $[r_1,R_1]\times J$, the function $\eta$ satisfies
\[
 \eta_t=d\eta_{rr}+\frac{d(N-1)}r\eta_r+c(r,t)\eta,
 \qquad c(r,t)=\kappa\mathbf1_{\{\eta(r,t)>0\}}.
\]
The coefficients are bounded and $\eta(r_1,t)\neq0$ and $\eta(R_1,t)=0$  for $t\in J$. Therefore
we can  apply the zero number theorem
\cite{Angenent1988,Chen1998}, in the form of
\cite[Proposition 3.1]{LouZhou2024}, which implies that
$\eta(\cdot,t)$ has finitely many zeros on $[r_1,R_1]$ for every
$t\in J$. This contradicts $\eta=0$ on $[R,R_1]\times J$.

Choose $\tau>0$ and a nonzero radial function
$\psi\in C_c^\infty(\Omega_R)$ with $0\le\psi\le\eta(\cdot,\tau)$. We construct a function $w$ which solves the following problem:
\begin{equation}\label{eq:exterior-neumann-heat}
 \begin{cases}
 w_t=d\Delta w,&x\in\Omega_R,\ t>0,\\
 \partial_\nu w=0,&x\in\partial B_R,\ t>0,\\
 w(x,0)=\psi(x),&x\in\Omega_R,
 \end{cases}
\end{equation}
where $\nu=-x/R$ is the outward unit normal of $\Omega_R$ on
$\partial B_R$. Let $w_j$ be solution of $w_t=d\Delta w$ on $B_{jR}\setminus B_R$ with homogeneous Neumann conditions on both boundaries and $w_j(x,0)=\psi(x)$, where $j$ is an integer and $\operatorname{supp}\psi(x)\subset B_{jR}\setminus B_R$.  By the Schauder estimates and  diagonal argument,  we have a sequence of $\{w_j\}$ locally uniform convergence to a bounded  nonnegative solution $w$ of \eqref{eq:exterior-neumann-heat} on $\Omega_R$ and   $\partial_\nu w(\cdot,t)=0$ on $\partial B_R$.

By Fatou's lemma,
$\int_{\Omega_R}w(x,t)\,dx\le P_0$, where
$P_0:=\int_{\Omega_R}\psi\,dx>0$.
Choose a smooth cutoff $\chi_A$ equal to one on $B_A$ and zero
outside $B_{2A}$, where $A>2R$, such that $0\le\chi_A\le1$ and
$|\Delta\chi_A|\le C/A^2$. Multiplying the equation in \eqref{eq:exterior-neumann-heat} by $\chi_A$
and integrating in space and time, we obtain
\[
 \left|\int_{\Omega_R}(w(x,t)-\psi(x))\chi_A(x)\,dx\right|
 \le CdP_0t/A^2.
\]
The boundary terms vanish since $\partial_\nu w=0$ and
$\nabla\chi_A=0$ near $\partial B_R$. Letting $A\to\infty$ yields
\begin{equation}\label{eq:neumann-positive-mass}
 \int_{\Omega_R}w(x,t)\,dx=P_0.
\end{equation}

By \eqref{eq:heat-exterior-sign},
$v(x,t)=e^{-\kappa t}\eta(x,\tau+t)$
satisfies $v_t=d\Delta v$ on
$\Omega_R$ and  $\partial_\nu v=-e^{-\kappa t}\eta_r(R,\tau+t)\ge0.$ Since  $w(x,0)=\psi\le\eta(\cdot,\tau)=v(x,0)$,
the comparison principle gives $v\ge w.$ Consequently,
\begin{equation}\label{eq:exterior-positive-mass-lower}
 \int_{\Omega_R}\eta(x,\tau+t)\,dx\ge e^{\kappa t}P_0\ge P_0.
\end{equation}
This remains valid when the integral on the left is infinite.

Let $(\lambda_1,\varphi_1)$ be the first Dirichlet eigenpair of
$-\Delta$ on $B_R$, with $\varphi_1>0$ and
$\max\varphi_1=1$. By continuity,
$\operatorname{supp}\eta_-(\cdot,0)\subset\overline B_{R_0}$.
Since $\min_{\overline B_{R_0}}\varphi_1>0$, we can choose $M>0$
such that $\eta(\cdot,0)\ge-M\varphi_1$ on $B_R$.
The function $-Me^{-d\lambda_1t}\varphi_1$ solves
\eqref{eq:heat-positive-part}  and has zero boundary values on $\partial B_R$. Since $\eta\ge0$ on $\partial B_R$,
 the comparison principle  gives $-\eta=\eta_--\eta_+\leq Me^{-d\lambda_1t}\varphi_1$, we have
\begin{equation}\label{eq:negative-mass-decay}
 N(t):=\int_{\mathbb R^N}\eta_-(x,t)\,dx
 \le Me^{-d\lambda_1t}\int_{B_R}\varphi_1(x)\,dx\to0.
\end{equation}
Choose $s\ge\tau$ with $N(s)<P_0/2$. By
\eqref{eq:exterior-positive-mass-lower}, there is $A>R$ such that
\[
 P:=\int_{B_A}\eta_+(x,s)\,dx>3P_0/4>N(s).
\]
Let $h$ satisfies $h_t=d\Delta h$ on $\mathbb R^N\times[0,\infty)$  with  $h(\cdot,0)=\eta(\cdot,s)$.
Since $\eta(\cdot,s)$ is bounded and $\kappa\eta_+\ge0$,  the comparison principle
and the heat kernel formula give, for $|x|\leq L$
\begin{align*}
 (4\pi dt)^{N/2}\eta(x,s+t)
 &\ge(4\pi dt)^{N/2}h(x,t)\\
 &=\int_{\mathbb R^N}e^{-|x-y|^2/(4dt)}(\eta_+(y,s)-\eta_-(y,s))\,dy\\
 &\geq\int_{B_A}e^{-|x-y|^2/(4dt)}\eta_+(y,s)-\int_{\mathbb R^N}e^{-|x-y|^2/(4dt)}\eta_-(y,s)\,dy\\
 &\ge e^{-(L+A)^2/(4dt)}P-N(s).
\end{align*}
Since $P>N(s)$, there exists $T_L>0$ such that
$\eta(x,t)>0$ for $|x|\le L$ and $t\ge T_L$.
This proves the lemma.
\end{proof}

\begin{theorem}\label{thm:exclude-theta}
Assume in addition \textup{(F2)} and \textup{(F3)}.
Then $\theta\notin\omega(u)$, and every
transition solution converges to a member of $\cS$ locally uniformly
in $\mathbb R^N$ as $t\to\infty$.
\end{theorem}

\begin{proof}
If $\theta\in\omega(u)$, Proposition~\ref{prop:quasi} shows that
$u(\cdot,t)$ converges to $\theta$ locally uniformly in $\mathbb R^N$
as $t\to\infty$.
By Lemma~\ref{lem:normalized-window-limit}, there is a sequence
$\tau_n\to\infty$ for which the functions $\eta_n$ converge to the
nonzero function $\eta$. Fix $L>R_0$. Lemma~\ref{lem:heat-mass}
gives a finite $T>0$ such that
$\min_{\overline B_L}\eta(\cdot,T)>0$. The convergence of $\eta_n$
on $\overline B_L$ at time $T$ then implies
\[
 u(x,\tau_n+T)>\theta\qquad(|x|\le L)
\]
for all large $n$. Since $u(\cdot,\tau_n+T)$ has compact support and satisfies
\eqref{eq:outermono}, $u(\cdot,\tau_n+T)-\theta$ has exactly one
zero outside $B_L$, and its radial derivative is negative at this
zero. This contradicts Lemma~\ref{lem:one-intersection}.
The assertion for transition solutions now follows from
Theorem~\ref{thm:conv}.
\end{proof}

\section{Sufficient conditions for vanishing and spreading}
\label{sec:vanishing-spreading}

\begin{lemma}\label{lem:belowtheta}
If $u_0$ is bounded and compactly supported and $0\le u_0\le\theta$,
then $u(\cdot,t)$ converges to zero uniformly in $\mathbb R^N$ as $t\to\infty$.
\end{lemma}

\begin{proof}
By comparison, $u\le\theta$, and so $u_t=\Delta u^m$. Choose a ZKB solution $B(\cdot,t_0)>u_0$.
Since $B(\cdot,t_0+t)\to0$ uniformly in
$\mathbb R^N$ as $t\to\infty$, the conclusion follows from
the comparison principle.
\end{proof}

\begin{proposition}\label{prop:nucleus}
There exists a continuous compactly supported function $\eta$, with
$0\le\eta<1$, such that spreading occurs if
$u_0(x)\ge\eta(x-a)$ for all $x\in\mathbb R^N$ and some
$a\in\mathbb R^N$.
\end{proposition}

\begin{proof}
Choose $\delta\in(0,1-\theta)$ and a smooth function $\chi$ with
$0\le\chi\le1$, $\chi=0$ on $(-\infty,\theta]$, $\chi>0$ on
$(\theta,\infty)$, and $\chi=1$ on $[\theta+\delta,\infty)$.
Then $\widetilde f=\chi f\in C^1([0,\infty))$, $\widetilde f\le f$,
$\widetilde f'(1)=f'(1)<0$, and
$\int_0^1s^{m-1}\widetilde f(s)\,ds>0$.
By \cite[Theorem 1.1]{Garriz2020}, there is a continuous compactly
supported function $0\leq\eta(x)<1$ such that the solution of  $u_t=\Delta u+\widetilde f(u)$ with $u(x,0)=\eta(x)$ spreads. The comparison principle and Theorem \ref{thm:conv} imply that the solution $u$ of \eqref{eq:PDE} with $u_0\geq \eta$  spreads.
\end{proof}

Fix $\phi\in\mathscr X$, and let $u_\sigma$ be the solution with
initial datum $\sigma\phi$. Denote by $\Sigma_0$ and $\Sigma_1$
the sets of parameters $\sigma>0$ for which vanishing and spreading
occur, respectively. Thus $u_\sigma(\cdot,t)\to0$ uniformly in
$\mathbb R^N$ for $\sigma\in\Sigma_0$, and
$u_\sigma(\cdot,t)\to1$ locally uniformly in $\mathbb R^N$ for
$\sigma\in\Sigma_1$, as $t\to\infty$.

\begin{lemma}\label{lem:open}
There are $0<\sigma_-\le\sigma_+\le\infty$ such that
\[
 \Sigma_0=(0,\sigma_-),\qquad \Sigma_1=(\sigma_+,\infty),
\]
where $(\infty,\infty)=\varnothing$. Every finite parameter in
$[\sigma_-,\sigma_+]$ is a transition parameter.
\end{lemma}

\begin{proof}
For sufficiently small $\sigma>0$, we have $\sigma\phi\le\theta$.
Lemma~\ref{lem:belowtheta} implies $\Sigma_0\ne\varnothing$. If $\sigma\in\Sigma_0$, there exists $T>0$ such that
$\|u_\sigma(\cdot,T)\|_\infty<\theta$.
Continuous dependence on the initial datum implies
$u_{\widehat\sigma}(\cdot,T)<\theta$ if  $|\widehat\sigma-\sigma|$ is sufficiently small and
comparison principle implies $u_{\tilde{\sigma}}(\cdot,T)<\theta$ for $\tilde{\sigma}<\sigma$. Then Lemma~\ref{lem:belowtheta} shows that  $\Sigma_0$ is a connected open interval $(0,\sigma_-)$.

If $\sigma\in\Sigma_1$, then   $u_\sigma$ sreads. Since the  function $\eta$ in Proposition \ref{prop:nucleus} has compact support,  there exists $T>0$ such that $u_\sigma(\cdot,T)>\eta$ . Continuous dependence on the initial datum implies
$u_{\widehat\sigma}(\cdot,T)>\eta$ if  $|\widehat\sigma-\sigma|$ is sufficiently small and
comparison principle implies $u_{\tilde{\sigma}}(\cdot,T)>\eta$ for $\tilde{\sigma}>\sigma$. Then Proposition \ref{prop:nucleus} shows that  $\Sigma_0$ is a connected open interval $(\sigma_+,\infty)$.

Since the two intervals are disjoint, we have
$\sigma_-\le\sigma_+\le\infty$. If $\Sigma_1$ is empty, we
set $\sigma_+=\infty$. Then for every finite parameter $\sigma\in[\sigma_-,\sigma_+]$, $u_\sigma$ a transition solution.
\end{proof}

\section{Sharp thresholds}
\label{sec:sharp-threshold}

If $\sigma_*\in[\sigma_-,\sigma_+]$, to prove uniqueness of the transition parameter, we compare
$u_\sigma$  with  $w_{\sigma_*}$  for $\sigma>\sigma_*$, where $w_{\sigma_*}$ is the   translation and rescaling of $u_{\sigma_*}$, as in
\cite[Lemmas 5.4 and 5.5]{LouZhou2024}.

\begin{lemma}\label{lem:translation-comparison}
For $i=1,2$, let $u_i$ be the solution of \eqref{eq:PDE} with initial datum
$\sigma_i\phi$, where $\phi\in\mathscr X$ and
$0<\sigma_1<\sigma_2$. There exists $T_0>0$ and $\varepsilon>0$
such that
\begin{equation}\label{eq:translation-comparison}
 u_1(x+a,t)\le u_2(x,t)
 \qquad(x\in\mathbb R^N,\ t\ge T_0,\ |a|\le\varepsilon).
\end{equation}
In particular,
\begin{equation}\label{eq:strict-support-gap}
 b_2(t)\ge b_1(t)+\varepsilon\qquad(t\ge T_0).
\end{equation}
\end{lemma}

\begin{proof}
Since $\sigma_2\phi>\sigma_1\phi$, the comparison principle and the strong maximum principle  give that for $u_1>0$,
\begin{equation}\label{eq:strict-interior-comparison}
 u_1(x,t)<u_2(x,t)\quad(t>0,\ |x|<b_1(t)),
 \qquad b_1(t)\le b_2(t).
\end{equation}

Choose $1<a_0<\sigma_2/\sigma_1$, put
$M=\max\{1,\|\sigma_1\phi\|_\infty\}$, and let $L$ be a Lipschitz
constant of $f$ on $[0,a_0M]$. Set
\[
 K>L(ma_0^{m-1}+1),\qquad
 a(t)=1+(a_0-1)e^{-Kt},\qquad
 s(t)=\int_0^t a(\xi)^{m-1}\,d\xi.
\]
Since $f(0)=0$ and $f$ is Lipschitz on $[0,a_0M]$, for
$0\le z\le M$ and $1\le a\le a_0$, we have
\[
 a^mf(z)-f(az)\le La^mz\leq L(ma_0^{m-1}+1)(a-1)z.
\]
Set $w(x,t)=a(t)u_1(x,s(t))$. Changing variables in the very
weak formulation \eqref{eq:very-weak} gives, in the sense of distributions,
\[
 w_t-\Delta w^m-f(w)
 = a'(t)z+a(t)^mf(z)-f(a(t)z)\le0,
 \qquad z=u_1(x,s(t)).
\]
Thus $w$ is a subsolution with $w(\cdot,0)\le\sigma_2\phi$.
By comparison, $w\le u_2$, and in particular
\begin{equation}\label{eq:time-rescaled-support}
 b_2(t)\ge b_1(s(t))\qquad(t\ge0).
\end{equation}
Since $s(t)>t$ and $b_1$ is nondecreasing, we have $b_1(s(t))\geq b_1(t)$. We claim that $b_1(s(t))>b_1(t)$ for some $t>0$. Otherwise,
$b_1(s(t))=b_1(t)$ for every $t>0$.  Since $s(t)-t$ is increasing, for $t_1>0$, the sequence
$t_{j+1}=s(t_j)$ satisfies
$t_{j+1}-t_j\ge s(t_1)-t_1>0$.
Then $t_j\to\infty$ and $b_1(t_j)=b_1(t_1)$,  which contradicts to
\eqref{eq:support-divergence}. Hence there exists $T_0>0$ such that
$b_2(T_0)\ge b_1(s(T_0))>b_1(T_0)$.

Choose $\rho>0$ such that $b_1(T_0)+2\rho<b_2(T_0)$, and set
$D=\overline B_{b_1(T_0)+\rho}$. Since $u_2>0$
on $D\setminus B_{b_1(T_0)}$, together with \eqref{eq:strict-interior-comparison}, we have
$d:=\min_D(u_2-u_1)(\cdot,T_0)>0$.
By uniform continuity, there exists $0<\varepsilon<\rho$ such that
$|u_1(x+a,T_0)-u_1(x,T_0)|<d/2$ on $D$ for $|a|\le\varepsilon$.
The supports of all these translated functions lie in $D$. Hence
$u_1(x+a,T_0)\le u_2(x,T_0)$ on the whole space.
The comparison principle gives \eqref{eq:translation-comparison}
for all $t\ge T_0$. Take $|a|=\varepsilon$, we obtain
\eqref{eq:strict-support-gap} from \eqref{eq:translation-comparison}.
\end{proof}

\begin{lemma}\label{lem:scaledDom}
Under the assumptions of Lemma~\ref{lem:translation-comparison},
for every $T\ge T_0$ there exists $A_0>1$ such that
\begin{equation}\label{eq:scaledinitial}
 u_2(x,T)\ge A u_1(A^{-m/2}x,T)
 \qquad(x\in\mathbb R^N,\ 1<A\le A_0).
\end{equation}
\end{lemma}

\begin{proof}
Set $D=\overline B_{b_1(T)+\varepsilon/2}$,   \eqref{eq:strict-support-gap} and \eqref{eq:strict-interior-comparison}
imply $d:=\min_D(u_2-u_1)(\cdot,T)>0$.
By continuity, there exists $A_0>1$ such that for $1<A<A_0$,
\[
 A^{m/2}b_1(T)<b_1(T)+\varepsilon/2,\qquad
 \sup_D|Au_1(A^{-m/2}x,T)-u_1(x,T)|<d/2.
\]
Thus \eqref{eq:scaledinitial} holds on $D$. Outside $D$, $u_1=0$, thus  \eqref{eq:scaledinitial} holds on  the whole space.
\end{proof}

\begin{lemma}\label{lem:thetaSharp}
Let $u_i$ be the solution of \eqref{eq:PDE} with initial datum
$\sigma_i\phi$, where $\phi\in\mathscr X$ and
$0<\sigma_1<\sigma_2$. If
$u_1(\cdot,t)\to\theta$ locally uniformly in $\mathbb R^N$ as
$t\to\infty$, then $u_2$ cannot have the same limit.
\end{lemma}

\begin{proof}
Suppose that both solutions converge to $\theta$. By
\textup{(F2)}, there is $\delta_\theta\in(0,1-\theta)$ 
such that
\begin{equation}\label{eq:monotonef}
 f\text{ is nondecreasing on }[\theta,\theta+\delta_\theta].
\end{equation}
By \eqref{eq:outermono},
$\|u_1(\cdot,t)\|_\infty=\max_{|x|\le R_0}u_1(x,t)\to\theta$.
Choose $T\ge T_0$ in Lemma~\ref{lem:translation-comparison} such that $u_1(x,t)\le\theta+\delta_\theta/4$
for all $x$ and $t\ge T$, and choose $1<A\le A_0$ in
Lemma~\ref{lem:scaledDom} with
$A(\theta+\delta_\theta/4)<\theta+\delta_\theta$.
Define
\[
 w(x,s)=A u_1(A^{-m/2}x,T+A^{-1}s),\qquad s\ge0,
\]
so that
\[
 w_s=\Delta w^m+f(w/A),\qquad 0\le w<\theta+\delta_\theta<1.
\]
Since $f=0$ on $[0,\theta]$, it follows from
\eqref{eq:monotonef} that $f$ is nondecreasing on
$[0,\theta+\delta_\theta]$. Hence $f(w/A)\le f(w)$, so $w$ is a
very weak subsolution of \eqref{eq:PDE}. Using
\eqref{eq:scaledinitial} and the comparison principle, we obtain
\[
 u_2(x,T+s)\ge A u_1(A^{-m/2}x,T+A^{-1}s)
 \qquad(x\in\mathbb R^N,\ s\ge0).
\]
Letting $s\to\infty$, for a fixed $x$ gives $\theta\ge A\theta$.
This is impossible since $A>1$ and $\theta>0$.
\end{proof}

\begin{proposition}\label{prop:uniqueN}
If $N\ge3$, then $\sigma_-=\sigma_+$.
\end{proposition}

\begin{proof}
Suppose that $\sigma_-<\sigma_+$. for $\sigma_-<\sigma_1<\sigma_2<\sigma_+$, let $u_1,\ u_2$ be the solution of \eqref{eq:PDE} with initial datum
$\sigma_1\phi,\ \sigma_2\phi$, respectively. By Theorems~\ref{thm:conv} and~\ref{thm:exclude-theta},  $u_1,\ u_2$ converge to $U_1,U_2\in\cS$, respectively, locally uniformly in
$\mathbb R^N$ as $t\to\infty$.
By comparison, $U_1\le U_2$. Proposition~\ref{prop:stationary}
therefore gives $U:=U_1=U_2$.
Let $t\to\infty$ in \eqref{eq:translation-comparison} with
$a=\varepsilon e_1$ we have $U(x+\varepsilon e_1)\le U(x)$.
Taking $x=-\varepsilon e_1$, we obtain $U(0)\le U(\varepsilon)$,
contrary to the strict decrease of $U$.
\end{proof}

\begin{proposition}\label{prop:unique2}
If $N=2$, then $\sigma_-=\sigma_+$.
\end{proposition}

\begin{proof}
Suppose that $\sigma_-<\sigma_+$, for $\sigma_-<\sigma_1<\sigma_2<\sigma_+$, let $u_1,\ u_2$ be the solution of \eqref{eq:PDE} with initial datum
$\sigma_1\phi,\ \sigma_2\phi$, respectively. By Theorem~\ref{thm:conv}, $u_1,\ u_2$ converge to $\theta$
locally uniformly in $\mathbb R^2$ as $t\to\infty$, which contradicts to
Lemma~\ref{lem:thetaSharp}.
\end{proof}

\begin{proof}[Proof of Theorem~\ref{thm:sharp-threshold}]
By Propositions~\ref{prop:uniqueN} and~\ref{prop:unique2}, we may
set $\sigma_*:=\sigma_-=\sigma_+$. Lemma~\ref{lem:open} gives vanishing for
$\sigma<\sigma_*$ and spreading for $\sigma>\sigma_*$.
If $\sigma_*<\infty$, $\sigma_*\notin\Sigma_0\cup\Sigma_1$,
thus $u_{\sigma_*}$
is a transition solution. By Theorem~\ref{thm:conv}, $u_{\sigma_*}$
converges locally uniformly to a stationary solution. Thus $u_{\sigma_*}$ converges to $\theta$
when $N=2$. When $N\ge3$, Theorem~\ref{thm:exclude-theta} shows
that $u_{\sigma_*}$ converges to a ground state solution $U\in\cS$ and
Proposition~\ref{prop:uniform} shows that $u_{\sigma_*}$
converges to $U$ uniformly in $\mathbb R^N$
as $t\to\infty$.
If $\sigma_*=\infty$, then $\Sigma_0=(0,\infty)$, so vanishing
occurs for every finite positive $\sigma$. In either case,
$\sigma_*$ is uniquely determined by the endpoint of $\Sigma_0$.
Finally, for $N\ge3$,
\eqref{eq:main-omega-exclusion} follows from
Theorem~\ref{thm:exclude-theta} applied to each $\sigma\phi$.
\end{proof}

\end{document}